\documentclass[%12pt
]{amsart}

\usepackage{tikz}

 \usetikzlibrary{arrows,calc}
 \usepackage[export]{adjustbox}
 \usepackage{float}
 \usetikzlibrary{backgrounds,arrows.meta}

 \usepackage{chngcntr}
\usepackage[belowskip=-8pt,aboveskip=6pt]{caption}

\usepackage{pgfplots}
\pgfplotsset{compat=newest}
\usepgfplotslibrary{fillbetween}

\usepackage[utf8]{inputenc} % usually not needed (loaded by default)
\usepackage[T1]{fontenc}

\usepackage{amsmath}
\usepackage{amsfonts}
\usepackage{amssymb}
\usepackage{amsthm}
\usepackage{setspace}
\usepackage{amsrefs}

\usepackage{graphicx}

\usepackage{hyperref}
\hypersetup{hidelinks}

\title[] 
{Integrability of Koszul connections on complex vector bundles over domains in ${\mathbf C}^n$ 
}
\author[]{Xianghong Gong$^{\dag}$}

\date{\today}

 \address{Department of Mathematics,
 University of Wisconsin-Madison, Madison, WI 53706}
 \email{gong@math.wisc.edu}

\thanks{$^{\dag}$Partially supported by  NSF grant DMS-2349865}

\makeatletter
\newcommand*\bigcdot{\mathpalette\bigcdot@{.75}}
\newcommand*\bigcdot@[2]{\mathbin{\vcenter{\hbox{\scalebox{#2}{$\m@th#1\bullet$}}}}}
\makeatother

\newtheorem{thm}{Theorem}[section]
\newtheorem{cor}[thm]{Corollary}
\newtheorem{prop}[thm]{Proposition}
\newtheorem{lemma}[thm]{Lemma}

\theoremstyle{definition}

\newtheorem{defn}[thm]{Definition}

\newtheorem{rem}[thm]{Remark}

\renewcommand{\th}[1]{\begin{thm}\label{#1}}
\renewcommand{\eth}{\end{thm}}
\newcommand{\co}[1]{\begin{cor}\label{#1}}
\newcommand{\eco}{\end{cor}}
\renewcommand{\le}[1]{\begin{lemma}\label{#1}}
\newcommand{\ele}{\end{lemma}}
\newcommand{\pr}[1]{\begin{prop}\label{#1}}
\newcommand{\epr}{\end{prop}}

\newcommand{\ga}{\begin{gather}}
\newcommand{\ega}{\end{gather}}
\newcommand{\gan}{\begin{gather*}}
\newcommand{\egan}{\end{gather*}}
\newcommand{\al}{\begin{align}}
\newcommand{\eal}{\end{align}}
\newcommand{\aln}{\begin{align*}}
\newcommand{\ealn}{\end{align*}}
\newcommand{\eq}[1]{\begin{equation}\label{#1}}
\newcommand{\eeq}{\end{equation}}

\newcommand{\DD}[2]{\frac{\partial #1}{\partial #2}}
\newcommand{\f}[2]{\frac{#1}{#2}}

\newcommand{\ci}{~\cite}

\newcommand{\cc}{{\bf C}}
\newcommand{\nn}{{\bf N}}

\newcommand{\rr}{{\bf R}}

\newcommand{\ov}{\overline}

\newcommand{\RE}{\operatorname{Re}}

\renewcommand{\dbar}{\overline\partial}

\renewcommand{\dbar}{\overline\partial}

\newcommand{\cL}{\mathcal}

\newcommand{\all}{\alpha}

\newcommand{\del}{\delta}
\newcommand{\Del}{\Delta}
\newcommand{\var}{\varphi}
\newcommand{\e}{\epsilon}
\newcommand{\om}{\omega}
\newcommand{\Om}{\Omega}

\newcommand{\la}{\lambda}

\newcommand{\pd}{\partial}

\newcommand{\re}[1]{(\ref{#1})}
\newcommand{\rea}[1]{$(\ref{#1})$}
\newcommand{\rl}[1]{Lemma~\ref{#1}}

\newcommand{\rp}[1]{Proposition~\ref{#1}}
\newcommand{\rt}[1]{Theorem~\ref{#1}}

\newcommand{\rta}[1]{Theorem~$\ref{#1}$}

\newcommand{\db}{\dbar}

\newcounter{pp}
\newcommand{\bpp}{\begin{list}{$\hspace{-1em}(\alph{pp})$}{\usecounter{pp}}}
\newcommand{\epp}{\end{list}}

\newcounter{ppp}
\newcommand{\bppp}{\begin{list}{$\hspace{-1em}(\roman{ppp})$}{\usecounter{ppp}}}
\newcommand{\eppp}{\end{list}}

\def\beq{\begin{equation}}
\def\eeq{\end{equation}}

\begin{document}

\begin{abstract}We study invertible matrix solutions $A$    to the equation $A^{-1}\overline\partial A=\omega^{(0,1)}$ on a small open subset $U$ of the closure $\overline M$ of a domain $M\subset{\mathbf C}^n$, where $\omega^{(0,1)}$ is a matrix of $(0,1)$ forms on $\overline M$ satisfying 
the formal integrable condition $\overline{\partial}\omega^{(0,1)}=\omega^{(0,1)}\wedge\omega^{(0,1)}$. For a $C^2$ domain $M$   that is either strongly pseudoconvex or  has at least $3$ negative Levi eigenvalues at a boundary point   contained in $U$, we obtain existence and sharp regularity of the solutions.
\end{abstract}

 \maketitle

%\tableofcontents

\setcounter{thm}{0}\setcounter{equation}{0}

\section{Introduction}\label{sec1} 
 Let $M\subset\cc^n$ be a relatively compact domain  with boundary $\pd M\in C^2$. Let $E$ be a smooth (i.e. $C^\infty$) complex vector bundle   over the closure $\ov M$. Denote by $C^\infty(\ov M, E), C_{(k)}^\infty(E,\ov M)$ the  spaces of smooth sections of $E$ and $E$-valued smooth $k$-forms on $\ov M$, respectively.
A Koszul {\it connection} $D$ in $E$ of is a $\cc$-linear map 
from $
 C_{(k)}^\infty(\ov M,E)$ to $ C^\infty_{(k+1)}(\ov M,E)
$
satisfying,  for any   $f\in C_{(k)}^\infty(\ov M)$ and     $s\in C^\infty(\ov M,E)$,  
$$
D(fs)=df\otimes s +(-1)^kf\wedge Ds.
$$     The curvature $R$ of $D$ is defined as
$
R=D\circ D\colon C^\infty_{(k)}(\ov M,E)\to C^\infty_{(k+2)}(\ov M,E).
$
  Let $e=(e_1,\dots, e_{k_0})^t$ be a smooth local frame of $E$. Then
$$
De_i=\sum \omega^j_i\otimes e_j, \ De=\om e;\quad  D^2e_i=\sum \Omega^j_i\otimes e_j, \ D^2e=\Omega e,
$$
where $\om$ and $\Om$ are  the connection and curvature forms.  
If $\hat e=Ae$ is another frame, then
\eq{hatomom}
\hat \om =(dA)A^{-1} +A\om A^{-1}, \quad  \widehat\Om=A\Om A^{-1}.
\eeq
Decomposing into types,  $\om=\om^{(1,0)}+\om^{(0,1)}$ and
$\Om=\Om^{(2,0)}+\Om^{(1,1)}+\Om^{(0,2)}$, yields  
\eq{om02}
 \Om^{(0,2)}=\db\om^{(0,1)}-\om^{(0,1)}\wedge\om^{(0,1)}.
\eeq

The {\it  integrability problem} studied in this paper is to find a frame   $\hat e=Ae$ for which $\hat\om^{(0,1)}=0$, which is equivalent to solving
\eq{dbAA}
A^{-1}\db A+\om^{(0,1)}=0.
\eeq
If such frames $\hat e_k$ exist on an open covering $\{U_k\}$ of $\ov M$, then on overlap $U_j\cap U_k$ their transition matrices satisfy    $\db  A_{kj}=0$. In this case we say that $(E,\ov M,D)$ carries a   holomorphic vector bundle structure.    By \re{hatomom}-\re{om02},  the solvability of \re{dbAA} requires the {\it formal} integrability condition
\eq{Om02=0}
\db\om^{(0,1)}=\om^{(0,1)}\wedge\om^{(0,1)},\quad \text{i.e.}\quad \Om^{(0,2)}=0.
\eeq
Koszul and Malgrange~\ci{MR131882} proved that when the boundary is not considered,   the formal integrability condition is sufficient for the integrality; see also Kobayashi~\ci{MR0909698}*{Prop.~3.7, p.~17} for a proof using the Newlander-Nirenberg theorem.

Our main result establishes integrability and optimal boundary regularity  for the H\"older-Zygmund class (see Section 3 for definition) under the formal integrability condition on $M$.
\begin{thm}\label{vbint-}
Let $0<r\leq\infty$. 
Let $M$ be a domain in $\cc^n$ with $\pd M\in C^2$. Assume that   the Levi form of $\pd M$ on $T^{(1,0)}\pd M$ has either $n-1$ positive or at least three negative eigenvalues at  $\zeta_0\in\pd M$.  
Let $E$ be a complex vector bundle of class $Lip$ $($resp. $\Lambda^{r+1})$ over $\ov M$. Assume that a connection  $D\colon Lip(M,E)\to L_{(1)}^\infty(M,E)$  $($resp. $\Lambda^{r+1}(M,E)\to \Lambda^{r}_{(1)}(M,E))$ satisfies $\Om^{(0,2)}=0$ on $ M$ in the sense of distributions. Then $E|_U$ has a    holomorphic vector bundle structure, where $U$ is some open subset of $\ov M$ and $\zeta_0\in U$. Furthermore, the holomorphic structure is of class $C^{1/2}$ $($resp. $\Lambda^{r+1/2})$ with respect to the original $Lip$ $($resp. $\Lambda^{r+1})$ structure of $E|_U$.
\end{thm}

 Our theorem concerns the integrability of $(E,D)$ in the H\"older-Zygmund spaces $\Lambda^r$, which is the standard H\"older spaces $C^r$ when $r$ is not an integer; see Section 5 for definition. The theorem also treats  a Koszul connection $D$ that maps the space $ Lip(M,E)$ of Lipschitz sections  
 into  $ L^\infty_{(1)}( M,E)$, the space  of $E$-valued $1$-forms of class $L^\infty$ on $M$. Such a connection extends trivially to a Koszul connection on $E$ over $\ov M$.   The theory of $L^\infty$ connections, which is of independent interest, is required  even when the original connections are $C^\infty$ for the (strictly) $3$-concave case. \rt{vbint-}  is reduced to \rt{vbint}, and the reduction requires a careful discussion in Section~\ref{sec:int-formal} of   $L^\infty$ connections on Lipschitz bundles  and of  {\it admissible} frame changes for which the transformation law \re{hatomom} remains valid; see \rp{d2=0}.

\medskip

We now outline our approach. Equation \re{dbAA} may be viewed as a nonlinear $\db$-equation, and the integrability condition \re{Om02=0} is also nonlinear. 
 We   apply homotopy formulas to the study of \re{dbAA}. 
The use of integral representations to establish regularity of $\dbar$-solutions on strongly pseudoconvex domains in $\cc^n$ has a long history. Sup-norm estimate for $\db$-solutions to $(0,1)$-forms was proved by  Grauert and Lieb~\cite{MR273057} and  Henkin\ci{MR0249660}.   Kerzman~\cite{MR0281944} obtained  $L^p$ and $C^{\beta}$ estimates of $\db$-solutions for $(0,1)$-forms and all $\beta<1/2$.
  Lieb~\ci{MR283235} obtained the $L^\infty$ and the $C^\beta$ estimates of $\db$-solutions for $(0,q)$-forms.  
Henkin and Romanov~\cite{MR0293121}  proved the sharp $C^{1/2}$ estimate of 
$\db$-solutions for continuous $(0,1)$-forms.

We recall several important developments on integral representation formulas for  $\dbar$   that lead to derivative estimates on strongly pseudoconvex domains $D\subset\cc^n$. When $\pd D$ is of class $ C^{k+2}$,   Siu~\cite{MR330515} showed that the Henkin solution operator satisfies a $C^{k+1/2}$ estimate  for  $\db$-closed $(0,1)$ forms of class $C^k$, and for  $\db$-closed $(0,q)$-forms   with $q\geq1$,  Lieb and Range~\cite{MR597825} constructed a new  $\db$ solution operator   and established  the same estimate. When $\pd D$ is only $ C^2$,  a homotopy formula was constructed  in~\cite{MR3961327} using the Stein extension, producing   homotopy operators with $\Lambda^{r+1/2}$ estimate for $r>1$.    Shi and Yao~\cites{MR4688544,MR4861589} later developed   a homotopy formula   employing the Rychkov extension and obtained $\Lambda^{r+1/2}$ estimates for all $r>0$, as well as $H^{s+1/2,p}$ estimates  for $1<p<\infty$ and $s>1/p$ when $\pd D\in C^2$; they also showed that the  estimates hold  for all $s\in\rr$ when $\pd D$ is sufficiently smooth.  Very recently, Yao~\ci{yaoc2} further reduced the boundary smoothness assumption to $ C^2$ for all $s\in\rr$. We also refer the reader to references cited in~\cites{MR3961327,MR4688544,MR4861589,yaoc2}.

We also note that   Range and Siu~\ci{MR338450} proved the $C^\beta$ estimate for all $\beta<1/2$ for $\db$-solutions for continuous $(0,q)$-forms on the (real) transversal intersection of
  strictly pseudoconvex domains. Whether the optimal gain of $1/2$  derivative   estimates  holds for $\db$-solutions on such transversal intersections remains an open problem.
  Higher-order derivative estimates in this setting were obtained by Brinkmann~\ci{Br84}, Michel~\ci{MR928297}, and Michel-Perotti~\ci{MR1038709}.
 Peters~\ci{MR1135535}  constructed a homotopy formula for the weakly transversal intersection of strictly pseudoconvex domains and established higher-order estimates, though with a loss of derivatives.
 It would be interesting to determine whether the results mentioned above can be applied to the integrability problem for intersections of strictly pseudoconvex domains. We also remark that the integrability of almost CR vector bundles on strongly pseudoconvex hypersurfaces in $\mathbb C^n$ with $n\geq4$ was proved by Webster~\ci{MR1128608}, and sharper regularity results in the finitely smooth case were obtained by Gong-Webster~\cites{MR2742034,MR2829316}.
It remains an open question whether \rt{vbint-} holds if $\pd M$ has two negative Levi eigenvalues. This appears to be related to the unsettled embedding problem for strongly pseudoconvex local CR structures in $\rr^5$, investigated by Webster~\cite{MR0995504} and Gong-Webster~\ci{MR2868966}.

We organize the paper as follows. In Section 2, we examine the formal integrability condition in detail for connections on Lipschitz complex vector bundles and introduce the notion of admissible frame changes. Sections 3 and 4 reformulate the local homotopy formulas in~\cites{MR986248,MR4866351} 
 for shrinking domains. In Section 5, we recall  estimates established in \cites{MR986248,gong-shi-nn} and derive  new estimates on shrinking domains. Finally, in Section 6, we complete the proof of \rt{vbint-}  by establishing \rt{vbint}.

% Oknek and  M. Schneider. Holomorphic vector bundles on complex projective space, p. 137

\setcounter{thm}{0}\setcounter{equation}{0}
\section{Formal integrability, optimal regularity, and counter examples}\label{sec:int-formal}

%
% When the boundary $\pd M$ is only of class $C^2$, it becomes necessary to consider lower regular connections  $D$  near the concave part of the boundary, even if the given connection is $C^\infty$ on $\ov M$. Therefore, in this section, we discuss the formal integrability for such  connections and their admissible frame changes. 

Let $M\subset\cc^n$ be a relatively compact domain with $C^2$ boundary. We say that $E$ is a {\it Lipschitz complex vector bundle}, if there exits an open covering $\{U_k\}$ of $\ov M$ and local frames $e_k$ over each $U_k$ such that the transition functions $g_{jk}\colon U_j\cap U_k\to GL(k_0,\cc)$ are Lipschitz, satisfy  $g_{jk}=g_{kj}^{-1}$, and obey the cocycle condition $g_{ij}g_{jk}g_{ki}=I$.  We denote by $Lip(M), Lip(M,E),Lip_{(k)}(M,E)$, and $ Lip_{(p,q)}(M,E)$ the Lipschitz  spaces  of functions,   sections of $E$, $E$-valued $k$-forms, and   $E$-valued $(p,q)$-forms on $\ov M$. These spaces are defined analogously when the Lipschitz class is replaced by other classes, such as $L^\infty, C^r$, or  H\"older-Zygmund class $\Lambda^r$ defined in Section 5.

A Koszul connection $D$ on $E$ is said to be of class $L^\infty$ if   
$$
D\colon Lip(M,E)\to L_{(1)}^\infty (M,E)
$$ and, for every   $f\in Lip(M)$ and   $s\in Lip(M,E)$,
\eq{Dfs}
D(fs)= df \otimes s+fDs.
\eeq
The connection $D$ extends naturally to  $E$-valued   forms of all degrees. If $f$ is a $Lip$ section of  $k$-form, analogue of \re{Dfs} is
$$
D(fs)=df\otimes s+(-1)^kf\wedge Ds.
$$

To determine the curvature for $L^\infty$ connections on Lipschitz bundles, it is convenient  to restrict attention to the interior $M$ of $\ov M$. We   define the curvature operator $R$ of $D$  by
\eq{R=D2}
R=D\circ D\colon Lip_{loc}(M,E)\to \cL D'_{(2)}(M,E).
\eeq
where $ Lip_{loc}(M,E)$ denotes locally Lipschitz sections of $E|_M$, and $\cL D'_{(2)}(M,E)$ denotes   $E$-valued distributional  $2$-forms on $M$.

Let $e=(e_1,\dots, e_r)^t$ be a Lipschitz    local frame over $U\cap M$. Then  
$$
De_i=\sum \omega^j_i \otimes e_j, \quad \om^j_i\in L^\infty_{(1)}( U)
$$
and the matrix $\om=(\om^j_i)$  is   the connection 1-form of $D$ with respect to the  frame $e$. In matrix notation we simply write $De=\om e$.  
For the curvature, we compute $D^2e=D(\om e)=d\om \otimes e-\om\wedge De$, and therefore 
$$
D^2e=\Om e, \quad \Om=d\om-\om\wedge\om,
$$
where $d\om$ is in the sense of distributions. If $\hat e=Ae$ is another Lipschitz frame, then
\eq{hatomA}
\hat \om =(dA)A^{-1} +A\om A^{-1}, \quad  \widehat\Om=A\Om A^{-1}.
\eeq
Of course, both \re{R=D2} and \re{hatomA} require justification; this is provided in \rp{d2=0}.

Assuming   \re{R=D2} and \re{hatomA},  the connection $D$ on $Lip_{(p,q)}(E)$  decomposes as $D=D'+D''$ where
$$
D'\colon Lip_{(p,q)}(E)\to L^\infty_{(p+1,q)}(E), \quad D''\colon Lip_{(p,q)}(E)\to L^\infty_{(p,q+1)}(E).
$$
Accordingly,  $R=D'\circ D'+(D'\circ D''+D''\circ D')+D''\circ D''$. Writing $\om=\om^{(1,0)}+\om^{(0,1)}$ and
$\Om=\Om^{(2,0)}+\Om^{(1,1)}+\Om^{(0,2)}$, the condition $D''\circ D''=0$ yields the $(0,2)$  component identity
\eq{Om02=0}
\Om^{(0,2)}=\db\om^{(0,1)}-\om^{(0,1)}\wedge\om^{(0,1)}=0.
\eeq    

The transformation law \re{hatomA} and the invariance of   \re{Om02=0} under suitable frame changes will play a fundamental role in the proof   \rt{vbint} and in    reducing \rt{vbint-} from \rt{vbint}.   
 
\le{dist} Let $U$ be an open subset of $\cc^n$.  Let $A$ be an invertible $k_0\times k_0$ matrix of continuous functions on $U$. Assume that $\db A\in L_{loc}^1(U)$ in the sense of distributions. Then the following hold in the sense of distributions on   $U$.
\bpp
\item   $\db A^{-1}\in L_{loc}^1(U)$       and $\db A^{-1}=-A^{-1}(\db A)A^{-1}$. 
When $k_0=1$, $\db \log A=A^{-1}\db A$. \item  Assume further that $\db A\in L_{loc}^2(U)$. Then
\eq{db2=0}
\db((\db A)A^{-1})=-\db A\wedge\db A^{-1},
\eeq
where $ \db A^{-1}:=A^{-1}(\db A)A^{-1}\in L^2_{loc}$,  and the definition is justified by $(a)$.
\epp
\ele
\begin{proof}Take $\chi\in C^\infty_0(\cc^n)$ satisfying $\int\chi=1$. Let $\chi_\e(z)=\e^{-2n}\chi(\e^{-1}z)$ and $A_\e=\chi_\e\ast A$.    As $\e\to0$, $A_\e$, $ A_\e^{-1}$, and (when $k_0=1$) $\log A_\e$, converge to $A,A^{-1}, \log A$, respectively, in $L_{loc}^\infty(U)$, and $ \db A_\e$ converges to $\db A$ in $L^1_{loc}(U)$.    Moreover,
$$
\db A_\e^{-1}=-A_\e^{-1}(\db A_\e)A_\e^{-1}$$
and, when $k_0=1$,  $\db\log A_\e=A_\e^{-1}\db A_\e$. Let $\e\to0$. We obtain $(a)$. 

To verify \re{db2=0}, let $\psi$ be any $k_0\times k_0$ matrix of smooth $(n,n-2)$ forms with compact support in $U$. Assume  $\db A\in L_{loc}^2(U)$. On each compact subset of $U$, the $L^2$ norms of $\db A_\e$ and $\db A_\e^{-1}=A_\e^{-1}(\db A_\e)A_\e^{-1}$ are uniformly bounded.  Applying Stokes' theorem and the Cauchy-Schwarz inequality to $\int \db A_\e\wedge \db A_\e^{-1}\wedge\psi$,  we get
$$
\int (\db A)A^{-1}\wedge\db\psi=\lim_{\e\to0}\int (\db A_\e)A_\e^{-1}\wedge\db\psi=-\int \db A\wedge\db A^{-1}\wedge\psi.\qedhere
$$
\end{proof}
\begin{lemma}\label{om''}
Let $\om^{(0,1)}$ be a $k_0\times k_0$ matrix of $(0,1)$ forms on an open set $U$ of $\cc^n$. 
Assume that $\om^{(0,1)}\in L_{loc}^2(U)$ and $\db \om^{(0,1)}=\om^{(0,1)}\wedge\om^{(0,1)}$ in the sense of distributions.
Let  $A$ be an invertible  $k_0\times k_0$  matrix of continuous functions   on $U$ with $\db A\in L_{loc}^2(U)$. Then 
$$
\hat\om^{(0,1)}:=(\db A )A^{-1}+A\om^{(0,1)} A^{-1}
$$
is still in $L_{loc}^2(U)$ and $\hat\Om^{(0,2)}=A\Om^{(0,2)}A^{-1}$ in the sense of distributions on $U$.
\end{lemma}
\begin{proof}Write $\om^{(0,1)},\hat\om^{(0,1)}$ as $\om,\hat\om$. By assumption,  $\om \wedge\om$ is in $L^1_{loc}(U)$.  Let $\db A^{-1}$ be defined  in \re{db2=0} and hence    
$
(\db A)A^{-1}+A\db A^{-1}=0.
$
We compute
\aln{}
\hat\om\wedge\hat\om&=((\db A)A^{-1}+A\om A^{-1})\wedge(-A\db A^{-1} +A\om A^{-1})\\
&=-(\db A)\wedge\db A^{-1}  +(\db A)\wedge \om A^{-1}-A\om \wedge\db A^{-1} +A\om \wedge\om A^{-1},\\
\db\hat\om& =-\db A\wedge\db A^{-1}+\db A\wedge\om A^{-1}+A(\db\om)A^{-1}-A\om\wedge\db  A^{-1}.
\end{align*}
Subtracting the above expressions yields  $\db\hat\om-\hat\om\wedge\hat\om  =A(\db\om-\om\wedge\om) A^{-1}$. 
\end{proof}
\begin{defn}A frame change via $A$ on $U$ is {\it admissible}, if $A, A^{-1}\in C^0(U)$ and $\db A\in L^2_{loc}(U)$.
\end{defn}
\begin{prop}\label{d2=0}
Let $D$ be a connection for a complex vector bundle $E$. Let $e=(e_1,\dots, e_{k_0})^t$ be a local frame field of $E$ over $U$. Assume that $De=\om \otimes e$ and $\om\in L^2_{loc}(U)$ and $d\om\in L^1_{loc}(U)$. If $A$ is admissible, then $\hat\om=(dA)A^{-1}+A\om A^{-1}$ is still in $L^2_{loc}(U)$ and  the curvature  forms    satisfy $A\Om A^{-1}=\hat\Om$. 
\end{prop}
\begin{proof}Since $A $ and  $\db A$ are in $ L^2_{loc}(U)$, then $dA\in L^2_{loc}(U)$; see~\ci{MR1045639}*{Thm.~4.2.5}. Thus \re{db2=0}  and the proof  of \rp{om''} are valid when $\db$ is replaced by $d$ and $dA^{-1}$ is defined as $-A^{-1}(dA)A^{-1}$.
\end{proof}

\begin{defn}\label{def-int} Let $D$ be an $L^\infty$ connection on a complex vector bundle $E$ of Lipschitz class over $\ov M$. 
We say that $(E,D)$ (or simply $D$)  is \emph{formally integrable}, if, for each $x\in  M$, there exists a Lipschitz frame   $e$    on a neighborhood $U$ of $x$ such that the curvature component  $\Om^{(0,2)}$ vanishes on $U$. We say that $D$ is \emph{integrable} if    for each point $x\in \ov M$ there exists an admissible frame change $\hat e=Ae$    in a neighborhood $V$ of $x$ in $\ov M$ for which $\hat\om^{(0,1)}=0$ on $V\cap M$.  
\end{defn}
\begin{rem}For an admissible frame change $\hat e=Ae$ starting from a  Lipschitz frame $e$, $\hat e$ may no long be of Lipschitz class with respect to the original Lipschitz bundle structure on $E$, unless the matrix $A$ itself is of Lipschitz class.
\end{rem}
%In summary, \nrc{d2=0} implies that the formal integrability of $D$ does not depend on the choice of frame fields that are Lip. 
%Note that by \re{hatomA}, condition $\om^{(0,1)}=0$ is preserved if and only if the transition function $A$ is holomorphic on the domain of $A$  that is contained in $M$.
% 

%
% 
%Let $M$ be a relatively compact domain in a complex manifold $X$ with boundary $\pd M\in C^2$. Let $E$ be a complex vector bundle of H\"older-Zygmund class $\Lambda^{r+1}$ with $r>0$ over the closure $\ov M$
%For an open  $U$ in  $\ov M$ and $0<a\leq r+1$, let $\Lambda^a(U)$ (resp. $\Lambda^a_{k}(U)$,  $\Lambda_{(p,q)}^a(U)$, $\Lambda^a(U,E)$, 
% $\Lambda^a_{k}(U,E)$,  $\Lambda_{(p,q)}^a(U,E)$) be the space of functions (resp. $k$-forms, $(p,q)$-forms, sections, $E$-valued $k$-forms, $E$-valued $(p,q)$ forms) on $U$ of class $\Lambda^a$. Note that when $a$ is not an integer, $\Lambda^a$ is the standard H\"older class $C^a$.
% See Section 3 for the definition of class $\Lambda^a$.

We conclude the section with the following results  showing   that \rt{vbint-} is optimal, except the open problem mentioned in Section~1.
\begin{prop}If the rank of $E$ is one, \rta{vbint-} and \rta{vbint} below remain valid when $M$ has  {\it two}
 negative Levi eigenvalues at $\zeta_0\in \pd M$. 
\end{prop}
\begin{proof}When the rank of $E$ is 1, the connection form $\om$ is simply a $(0,1)$-form. In this case $\om\wedge\om=0$, and the  formal integrability is reduced to the condition $\db\om=0$ on $U$. By \rt{concave-est}, there exists a solution $u$ in some neighborhood $U$ of $\zeta_0$ such that $\db u=\om$ on $U\setminus M$. Moreover,  $u\in C^{1/2}(U)$ when $\om\in L^\infty$, and $u\in\Lambda^{r+1/2}$ when $\om\in\Lambda^r$. Set $A=e^u$. Then $A$ is continuous on $U$ and satisfies $A\neq0$. On $U\setminus \pd M$, we have $\db\log A=\om$ and $\db A=A\om$ by \rl{dist}. 
\end{proof}
The following example shows   that  in general one cannot  expect a holomorphic structure for $(E,D)$   that  is   of class $C^{r+\e}$ ($\e>1/2$) with respect to the original $C^{r+1}$ bundle structure equipped with  a connection, i.e. a connection form,  of class $C^r$. Let $B^n$ denote the unit ball in $\cc^n$.
\begin{prop}Let $n\geq2$ and $\zeta_0=(1,0')\in\pd B^n$. There exists a $\db$-closed $(0,1)$-form $\om$ on $B^n$  with $\om\in C^r(\ov{B^n})$  such that for any neighborhood $U$ of $\zeta_0 $ in $\ov{B^n}$, $A$ is not in $C^{r+\all}(U)$ for every $\all>1/2$, if $\db A=A\om$ holds in the sense of distributions on $U\setminus\pd B^n$ and $A$ is non-vanishing and continuous  on  $U$.
\end{prop}
\begin{proof}We adapt Stein's classical example~\ci{MR0774049}*{p.~73} for the linear equation $\db u=\om$. Let $r\geq0$. Choose the branch $\arg\log (z_1-1)  \in[0,2\pi)$ and define
 $$
(z_1-1)^a:=e^{a\log(z_1-1)}, \quad a\in\rr;\qquad  \om(z):=\frac{(z_1-1)^r}{\log (z_1-1)}d\ov z_2.
 $$
Then $\om\in C^r(\ov{B^n})$, and $\db \om=0$;  hence $\db\om-\om\wedge\om=0$ on $B^n$.  

We first claim that if $\db u=\om$ then  $u$ cannot be in $ C^{r+\all}(\ov{B^n}\cap U)$ for any $\all>1/2$. Otherwise, replacing $\all$ by a smaller number   still greater than $1/2$, we may assume   $r+\all=k+\beta$, where $0<\beta<1$,  $k=[r]$  when $r-[r]<1/2$ or $k=[r]+1$  when $r-[r]\geq1/2$.    Then $u(z)-\frac{(z_1-1)^r\ov z_2}{\log(z_1-1)}$ is  holomorphic on $U\cap B^n$. Set  $u_k(z):=\pd^k u(z)/{\pd z_1^k}$ and $$
Q(z_1):=\frac{d^k}{dz_1^k}\f{(z_1-1)^{r}}{\log(z_1-1)}=(z_1-1)^{r-k}\sum_{j=1}^{k+1}\f{ C_j}{(\log(z_1-1))^{j}}, \quad C_1=\binom{r}{k}\neq0.
$$   
Then $u_k-Q(z_1)\bar z_2$ is holomorphic on $U\cap B^n$. By the Cauchy formula,
$$
\int_{|z_2|=\sqrt{|\e|}}(u_k(1-\e,z_2)-u_k(1-2\e,z_2))\, dz_2=( Q(-\e)-Q(-2\e))\int_{|z_2|=\sqrt\e}\bar z_2\, dz_2,
$$
 for $0<\e<1/4$. We estimate
$$
|C_1|\e\left|\frac{(-\e)^{r-k}}{\log(-\e)}-\f{ (-2\e)^{r-k}}{\log(-2\e)}\right|-\sum_{j=2}^{k+1}\f{C_j'|\e|^{r-k+1}}{|\log\e|^{j}}\leq \|u\|_{C{r+\all}}\e^{r+\all-k+1/2}.
$$
Treating separately the cases $r-k=0$ and $r-k>0$,  we obtain $\f{1}{2}|C_1||\log\e|^{-2}\leq \|u\|_{C^{r+\all}}\e^{\all-1/2}$, which is impossible for $\all>1/2$ and small $\e>0$.  This proves the claim.

Now suppose that $A\in C^{r+\all}(U)$  for some $\all>1/2$, $A$ does not vanish,  and $\db A=A\om$ on $U\subset B^n$. By \rl{dist},  we obtain  $\db (\log A)=A^{-1}\db A=\om$ in the sense of distributions, where $\log A\in C^{r+\all}(U)$ after shrinking $U$ if necessary. This contradicts the  claim. 
 \end{proof}
Finally, we give an example showing that \rt{vbint-} and \rt{vbint} fail when $\pd M$ has only one negative Levi eigenvalue.
\begin{prop}Let $H\subset\cc^2$ be the Heisenberg group defined by  $y_2=|z_1|^2$, and let $H^-=\{y_2<|z_1|^2\}$. There exists a smooth function $f(z_1,x_2)$ such that for every $\zeta_0\in H$ and every open subset $U\subset \ov{ H^{-}}$ containing $\zeta_0$,  the $(0,1)$ form $f(z_1,x_2)d\ov z_1$ on $H$ extends to a $\db$-closed smooth $(0,1)$ form $\om$ on $U$, for which the equation $\db A=A\om$
admits {\it no}  $C^1$ solution on $H\cap U_0$  with $A(\zeta_0)\neq0$, where $U_0$ is any neighborhood of $\zeta_0$ in $\ov {H^-}$.  
\end{prop}
\begin{proof}Let $L=\partial/{\partial\bar z_1}-z_1\partial/{\partial x_2}$. By  Lewy's classical non-solvability theorem, there exists    a smooth function $f$ on $H$ such that the equation $Lu=f$ has no continuous solution on any non-empty open subset of $H$.  Andreotti and Hill~\ci{MR0460725}*{\S 5} showed that for every $\zeta_0\in H$,  the $(0,1)$ form  $f(z_1,x_2)\, d\ov z_1$ on $H$ extends to a $\db$-closed $(0,1)$ form $\om$ on some neighborhood $U^-$ of $\zeta_0$ in $H^-$.  Suppose, for the sake of contradiction, that there exists a $C^1$ solution $A$  of $\db A=A\om$ on some neighborhood $U_0\subset H^-$ of $\zeta_0$, with $A\neq0$. Then  $A^{-1}\db A=\om$, so setting $u=\log A$ we obtain    $\db_b u=f(z_1,x_2)\, d\ov z_1$ on $U'\cap H$, where $\db_b$ is the tangential Cauchy-Riemann operator on $H$. But $\db_b u=fd\ov z_1$ is equivalent to  $Lu=f$,  a contradiction.
\end{proof}

\setcounter{thm}{0}\setcounter{equation}{0}

\section{Local homotopy formula for $(n-q)$ convex domains}\label{sect:convex}
 
In this section, we recall a local homotopy formula near a  strictly $(n-q)$ convex boundary point, as constructed in~\cites{MR986248,MR4866351,gong-shi-nn}. We  derive the formula on an original domain $D^{12}_{r}$ where the forms are defined. 
Our notation follows that of~\cites{MR4866351,MR986248}.
Throughout paper, we use $E^{1\dots\ell}=E^1\cap\cdots\cap E^\ell$ for sets $E^1,\dots, E^\ell$. 
\begin{defn}\label{pck}  
%2/18
 Let $D^1,\dots, D^\ell$ be open sets in $\cc^n$.
We say that  $D^{1\dots\ell}$ is a 
  {\it $($real$)$ transversal intersection} of $D^1,\dots, D^\ell$, if  $\ov{D^{1\dots\ell}}$ is compact, there exist $C^1$  real-valued functions $\rho^1,\dots,\rho^\ell$ on a neighborhood $U$ of  $\ov{D^{1\dots\ell}}$ such that $D^j= U\cap\{\rho^j<0\}$, and for every $1\leq j_1<\cdots<j_i\leq\ell$,
$$
d\rho^{j_1}(\zeta)\wedge\cdots\wedge d\rho^{j_i}(\zeta) \neq0, \quad\text{whenever $\rho^{j_1}(\zeta)=\cdots=\rho^{j_i}(\zeta)=0$}.
$$
\end{defn}
%From the definition, we see that $D^{1\dots\ell}=D^1\cap\cdots\cap D^\ell$. 
Let $S^j:=\{\rho^j=0\}\cap \pd D^{1\dots\ell}$.  Each  $S^j$ carries a unique orientation such  that the Stokes' formula takes the form
\eq{orientations}
\int_{D^{1\cdots\ell}}df=\int_{S^1}f+\cdots+\int_{S^\ell}f.
\eeq

\begin{defn} Let $D$ be a domain in $\cc^n$, and let $S\subset \cc^n\setminus D$ be a $C^1$ submanifold in $\cc^n$. A mapping $g\colon D\times S\to \cc^n$ is called a {\it Leray map}   if $g\in C^1(D\times S)$ and
$$%\eq{leray-map}
g( z,\zeta )\cdot(\zeta-z)\neq 0, \quad \forall (z,\zeta)\in D\times S.
$$%\end{equation}
\end{defn}

We always use the standard Leray mapping $g^0(z,\zeta)=\ov z-\ov\zeta$.
Let $g^j \colon D\times S^j\to\cc^n$ be   $C^1$ Leray   mappings for $j=1,\dots,\ell$.
Let $w=\zeta-z$ and define   \gan
\omega^{i}(z,\zeta)=\f{1}{2\pi i}\f{g^i(z,\zeta)\cdot dw}{g^i(z,\zeta)\cdot w},
\quad
\Omega^{i}=\omega^i\wedge(\ov\pd\omega^i)^{n-1},\\
\Omega^{i_1\cdots i_k}=\omega^{i_1}\wedge\cdots\wedge\omega^{i_k}\wedge\sum_{\alpha_1+\cdots+\all_k=n-k}
(\ov\pd\omega^{i_1})^{\alpha_1}\wedge\cdots\wedge(\ov\pd\omega^{i_k})^{\all_k}.
\end{gather*}
Decompose $\Om^{\bigcdot}=\sum\Om_{(0,q)}^{\bigcdot}$,
where $ \Om_{(0,q)}^{\bigcdot}(z,\zeta)$
  has type $(0,q)$  in the $z$-variable; thus $\Om^{i_1,\dots, i_\ell}_{(0,q)}$ has type $(n,n-\ell-q)$    in $\zeta$. For convenience, set
   $\Om^{\bigcdot}_{0,-1}=0$ and $\Omega_{0,n+1}^{\bigcdot}=0$. The Koppelman lemma says that
   $$
   \db_z\Om^{i_1\cdots i_k}_{(0,q-1)}+\db_\zeta\Om^{i_1\cdots i_k}_{(0,q)}=\sum(-1)^{j}\Om_{(0,q)}^{i_1\cdots\hat {i_j}\cdots i_k}.
   $$
See Chen-Shaw~\cite{MR1800297}*{p.~263} for a proof. For instance, we have 
\ga
\label{kop1}\db_\zeta\Om_{(0,q)}^{i_1}+\db_z\Om_{(0,q-1)}^{i_1}=0, %  \quad q\geq 0,
\\
\label{kop2}\db_\zeta\Om^{i_1i_2}_{(0,q)}+\db_z\Om_{(0,q-1)}^{i_1i_2}=-\Om_{(0,q)}^{i_2}+\Om_{(0,q)}^{i_1}, %\quad q\geq0,
\\
\label{kop3}\db_\zeta\Om^{i_1i_2i_3}_{(0,q)}+\db_z\Om_{(0,q-1)}^{i_1i_2i_3}=
-\Om_{(0,q)}^{i_2i_3}+\Om_{(0,q)}^{i_1i_3}-\Om^{i_1i_2}_{(0,q)}.
%, \quad q\geq0.
\end{gather}

% Following~\cite{MR1800297}*{p.~263},   
For a   function $u$ on a submanifold $M$ in $\cc^n$, define
$$
  \int_{y\in M} u(x,y)dy^J\wedge dx^I =  \left \{\int_{y\in M}u (x,y)dy^J\right\}dx^I.
$$
Then the exterior differential $d_x$ satisfies
\eq{checksign}
d_x\int_{y\in M}\phi(x,y) =(-1)^{\dim M}\int_{y\in M}d_x\phi(x,y).
\eeq

 Let $D\subset  U_0$ be defined by $\rho^0<0$ with $\rho^0\in C^2(U_0)$.
Suppose that the Levi-form of $\rho^0$ has   $(n-q)$ positive eigenvalues at $\zeta_0\in\pd D$.
Then there exist an open set $U_1\Subset U_0$ containing $\zeta_0$ and  a  biholomorphic mapping $\psi$ defined on a neighborhood of $\ov{U_1}$  such that $\psi(\zeta_0)=0$, $U:=\psi(U_1)$ is a polydisc, and  $D^1:=\psi(U_1\cap D)$ is defined by
\ga{}\label{qconv-nf}
D^1=\{z\in U\colon\rho^1(z)<0\},\\
\rho^1(z)=-y_{n}+\la_1|z_1|^2+\cdots+\la_{q-1}|z_{q-1}|^2+|z_{q}|^2+
\cdots+|z_{n}|^2+R(z),
\end{gather}
where $|\la_j|<1/4$ and $R(z)=o(|z|^2)$.  Furthermore, there exists $r_1>0$  such that the boundary  $\pd D^1$ intersects the sphere $\pd B_r$ transversally when $0<r<r_1$. See~\ci{MR4866351} for details.
Let
\eq{rho2}
\rho^2(z)=|z|^2-r^2_2,\quad D_{r_2}^2=\{z\colon\rho^2<0\},
\end{equation}
 where $0<r_2<r_1$ and $r_1/2<r_2<r_1$.  Define
\ga{}\label{d12} D_{r_2}^{12}=D^1\cap D_{r_2}^2,\quad
%\label{s12}
\pd D^{12}=S^1\cup S^2, \quad S^i\subset\pd D^i.
\end{gather}
Here and in what follows, we may drop the subscript in $D^2_{r_2},D^{12}_{r_2}$. We also define
\ga{}
\label{g1}
g^{1}_j( z,\zeta )=\begin{cases}
\DD{\rho^1}{\zeta_j},&q\leq j\leq n,\\
\DD{\rho^1}{\zeta_j}+(\ov\zeta_j-\ov z_j),& 1\leq j<q,\end{cases}
\\
\label{g02}
 g^2( z,\zeta )=\Bigl(\DD{\rho^2}{\zeta_1},\dots, \DD{\rho^2}{\zeta_n}\Bigr)=\ov\zeta.
\end{gather}
Then for $\zeta,z\in U$ and by shrinking $U$ if necessary, we have
\al{}\label{W1-dist}
2\RE\{ g^1( z,\zeta )\cdot(\zeta-z)\}&\geq \rho^1(\zeta)-\rho^1(z)+\f{1}{2}|\zeta-z|^2. 
% \\ |g^1(z,\zeta)\cdot(\zeta-z)|&\geq 1/C_*, \quad \forall\zeta\in S^{12}, z\in B_{\tilde r_2}\label{g1z}
\end{align}
Note that $g^2(z,\zeta)$ is holomorphic in $z$, while  $g^1(z,\zeta)$ is anti-holomorphic in only the first $q-1$ variables of $z$.  Thus, we have
\al{}
\label{type-1}
& \Omega_{(0,k)}^{1}( z,\zeta )=0, \  k\geq q;\quad \db_z\Om_{(0,q-1)}^1( z,\zeta )=0;
\\
\label{type-2}
 &\Omega_{(0,k)}^{2}( z,\zeta )=0, \  k\geq1;\quad \db_z\Om^2_{(0,0)}(z,\zeta)=0;\\
\label{type-12}
&\Omega_{(0,k)}^{12}( z,\zeta )=0,\  k\geq q;\quad \db_z\Om_{(0,q-1)}^{12}( z,\zeta )=0.
\end{align}

We now introduce the following integrals   on open sets in $\cc^n$: 
\ga \label{defnLR}
R_{D; q}^{i_1\dots i_\ell}f(z):=\int_{\zeta\in D}\Om^{i_1\dots i_\ell}_{(0,q)}(z,\zeta)\wedge f(\zeta).
\end{gather}
For convenience, we  also write this as $R_{D; q}^{i_1\dots i_\ell}f=\int_{ D}\Om^{i_1\dots i_\ell}_{(0,q)} \wedge f$. Similarly, we define integrals on low dimensional sets
\ga{}\label{defnLR+}
L_{i_1\cdots i_\mu; q}^{j_1\dots j_\nu}f:=\int_{S^{i_1\cdots i_\mu}}\Omega_{(0,q)}^{j_1\cdots j_\nu} \wedge f,
\end{gather}
where $S^{i_1\dots i_\mu}$ is the transversal intersection of $C^1$ real hypersurfaces $S^{i_1}, \dots,  S^{i_\mu}$.
%By  \re{type-2}-\re{type-12}, we have for on $D^{12}$
%\ga\label{L110}
%L_{i;q}^if=\int_{S^i}\Om_{(0,q)}^i\wedge f=0, \  i=1,2; \quad L_{12;q}^{12}=0;\\
% \db_z\int_{D^{12}}\Om_{(0,q-1)}^1(z,\zeta)\wedge f(\zeta)=0, \quad \int_{D^{12}}\Om_{(0,q)}^1\wedge \db f=0.
%\end{gather}
\begin{thm}[\cites{MR986248}]\label{cchf00} Let $ 0<q\leq n$.
Let $(D^1,D_{r_2}^2)$, defined by \rea{qconv-nf}-\rea{rho2},  be a $(n-q)$-convex configuration with   Leray maps $g^1,g^2$ defined by \rea{g1}-\rea{g02}.  Let $f\in
C^0_{(0,q)}(\ov{D_{r_2}^{12}})
$ with $\db f\in C^0(\ov{D_{r_2}^{12}})$. Then
$% \ga{}\label{tsqf+-cv0}
f=\db H^{(0)}_qf+  H^{(0)}_{q+1}\db f  %,\quad q>0,\\
%\label{tsqf+-cv}
% f=L_1^1f+L_2^2f+L_{12}^{12}f+  H_1\db  f.
$%\end{gather}
  on $D^{12}_{r_2}$, where  $H^{(0)}_s=R_{D_{r_2}^{12},s-1}^{(0)}-L^{01}_{1,s-1}-L^{02}_{2,s-1}+L^{012}_{12,s-1}$.
\eth 

To construct the homotopy formulas, we extend forms on $D^{12}_r$ to a larger domain.
We  use   the Rychkov extension operator $\cL E_{D}$ to extend functions on  a bounded Lipschitz domain $D\subset\cc^n$ to functions with compact support in $\cc^n$. The  operator $\cL E_{D}$ is applied componentwise for the coefficients of differential forms. We  also use the commutator $[\cL E_D,\db]f=\cL E_D(\db f)-\db(\cL E_Df)$ for a form $f$. Key estimates on $\cL E|_D$ will be stated in Section~\ref{h-space}.

Define 
$$
S^1_+=\pd D^2\setminus D^1, \quad U^1=D^2\setminus \ov{D^1}.
$$
 Thus, $D^{12} \cup S^1\cup U^1=D^2$. Note that $\pd U^1=S^1\cup S_+^1$. Define
$$%\eq{L1+}
L_{1^+; q}^{01}f=\int_{S^1_+}\Omega_{(0,q)}^{01}\wedge\cL E_{D^{12}_{r_2}} f.
$$%\eeq
%By  \re{type-2}-\re{type-12}, we have for on $D^{12}$
%\ga\label{L110-e} 
% \db_z\int_{U^1}\Om_{(0,q-1)}^1(z,\zeta)\wedge\cL  Ef(\zeta)=0, \quad \int_{U^1}\Om_{(0,q)}^1\wedge \cL E\db f=0.
%\end{gather}

We have   the following  homotopy formula in~\cites{MR4866351,gong-shi-nn}.
\begin{thm}\label{hf-c} Let $0<q\leq n$. Let $D^1,D^2_{r_2},g^1,g^2$ be as in \rta{cchf00}.
% Suppose that $U^1\subset\cc^n\setminus\ov{D^1}$, $\pd U^1=S^1\cup S^1_+$ and
Let $U^1=D_{r_2}^2\setminus\ov {D^1}$ and  $S_+^1=\pd D_{r_2}^2\setminus D^1$.
 Suppose that $  f$ is a $(0,q)$ form such that $f$ and $\db f$ are in $C^\e (\ov {D_{r_2}^{12}})$ with $\e>0$. Then 
$%\gan%\label{tsqf-c}
 f= \db  H_q f+  H_{q+1}\db f
$%\end{gather*}
  on  $D_{r_2}^{12}$, where    $H_s:=H_s^{(1)}+H_s^{(2)}$  and
\al{} 
\label{hq1} H^{(1)}_s &:=R_{ D_{r_2}^{2}; s-1}^0 \cL E_{D_{r_2}^{12}} +R_{U^1;s-1 }^{01}[\db,\cL E_{D_{r_2}^{12}}],\\
\label{nhq2-}
H^{(2)}_s&:=L_{1^+;{s-1}
}^{01} \cL E_{D_{r_2}^{12}}   -L_{2;s-1}^{02} +L_{12;s-1}^{012}.
\end{align}
\end{thm}

\setcounter{thm}{0}\setcounter{equation}{0}
\section{Local  
 homotopy formula for concave domains}\label{sect:concave}
 In this section, we recall a local homotopy formula for a strictly $(q+2)$ concave boundary point (i.e. the point at which the Levi-form has at least $q+2$ negative eigenvalues), as constructed in~\cites{MR986248,MR4866351,gong-shi-nn}.   We  also formulate the formula on shrinking domain $D^{123}_{(1-\theta)r}$, starting  from the original domain $D^{123}_{r}$ where forms are defined. 

Assume that $\pd D$ is strictly $(q+1)$ concave  at $\zeta_0\in\pd D$. Then there exist an open set $U_1\Subset U_0$ containing $\zeta_0$ and  a  biholomorphic mapping $\psi$ defined on a neighborhood of $\ov{U_1}$  such that $\psi(\zeta_0)=0$, $U:=\psi(U_1)$ is a polydisc, and $D^1:=\psi(U_1\cap D)$ is defined by  
\begin{gather}\label{rho1-v}
D^1=\{z\in U\colon\rho^1(z)<0\}, \\
\rho^1(z)=-y_{q+2}-\sum_{j=1}^{q+2}|z_j|^2 +\sum_{j=q+3}^n\la_{j}|z_{j}|^2 +R(z),
\end{gather}
where $|\la_j|<1/4$, 
 and $R(z)=o(|z|^2)$.  There exists $r_1>0$  such that the boundary  $\pd D^1$ intersects the sphere $\pd B_r$ transversally when $0<r<r_1$.

When $\rho^1$ has the form \re{rho1-v}, as in~\cite{MR986248}*{pp.~118-120} define
%$g^1( z,\zeta )=W^1_{\rho^1}(z,\zeta)$. Explicitly, we have
\eq{}\label{HL2pg81}
g^1_{j}( z,\zeta )=\begin{cases}
\DD{\rho^1}{z_j},&1\leq j\leq q+2,\\
\DD{\rho^1}{z_j}+\ov z_j-\ov \zeta_j,& q+3\leq j\leq n.
\end{cases}
\end{equation}
Then we have
\eq{W1s-dist}
-2\RE\{g^1( z,\zeta )\cdot(\zeta-z)\}\geq \rho^1(\zeta)-\rho^1(z)+|\zeta-z|^2/C.
\end{equation}
Note that $g^1(z,\zeta)$ is holomorphic in $\zeta_1,\dots, \zeta_{q+2}$.
We still use $\rho^2,D^2_{r_2}$, and $g^2$,  defined by  \re{rho2} and \re{g02}.
 Unlike the $(n-q)$ convex case,   a third domain is required:
\eq{defD3}
D^3\colon \rho^3<0, \quad 0\in D^3
\end{equation}
where $
\rho^3(z):=-y_{q+2}+\sum_{j=q+3}^n3|z_j|^2-r^2_3$
with $0<r_3<r_2/{C_n}$. Define
\ga
\label{defW-3}
g^{3}_j( z,\zeta )=\begin{cases}
0,&1\leq j<q+2,\\
i,&j= q+2,\\
3(\ov\zeta_j+\ov z_j),& q+3\leq j\leq n.
\end{cases}
\end{gather}
We can verify
\eq{HH}
\RE\{g^3(z,\zeta)\cdot(\zeta-z)\}=\rho^3(\zeta)-\rho^3(z).
\end{equation}
Denote by $\text{deg}_\zeta$  the degree of a form in $\zeta$. We have
\ga\label{type-3}
\text{deg}_\zeta\, \Om^3_{(0,\ell)}( z,\zeta )\leq n+([n-(q+3)+1]-\ell)= 2n-q-\ell-2, \quad\forall\ell.
\end{gather}

Then $D^{123}$ is a $C^1$ transversal intersection of $D^1, D^2, D^3$.
%See \rf{fig:concave} for relations of $D^1, D^2, D^3$.
 We call  $(D^1, D^2,D^3)$, defined by \re{rho1-v}, \re{rho2}, and \re{defD3},
 a {\it $(q+1)$-concave configuration}.  
  The $g^1,g^2,g^3$ in \rea{HL2pg81}, \rea{g02} and \rea{defW-3}  are called the {\it  Leray maps} of the configuration.
  When we need to indicate the dependence of $D^2,D^3$ on  $r_2,r_3$, we will write
\eq{Diri}
D^i_{r_i}=D^i, \quad D^{23}_{r}=D_{r_2}\cap D_{r_3}^3, \quad D^{123}_r=D^1\cap D^2_{r_2}\cap D^3_{r_3}.
\eeq

As  in  Stokes' formula \re{orientations}, we   define the following oriented boundaries:
\ga{} 
\pd D^{123}=S^1+S^2+S^3,\quad S^i=\{\rho^i=0\}\cap\ov{D^{123}},\\
\pd S^1=S^{13}, \quad \pd S^2=S^{23},\quad \pd S^3= S^{32}+S^{31}.
\end{gather}
With the above orientations,  we have
$$%\eq{}
\pd(D^3\setminus \ov{D^1})=S_+^1-S^1,\quad S^{ij}=-S^{ji}.
$$%\eeq

 Note that   \re{type-2} still holds for $\Om^2$ and 
\ga
\label{Ca-type-1}
\text{deg}_\zeta\,  \Omega_{(0,*)}^{1}(z,\zeta)\leq 2n-q-2.
\end{gather}
Therefore, we   have for a $(0,q)$ form $f$,
\ga{}\label{L110}L_{i;q}^if =\int_{S^i}\Om_{(0,q)}^i \wedge f =0, \quad  i=1,2.
\end{gather}
Note that
\ga{}
\label{db13=0}
\Om^{13}_{(0,\ell)}( z,\zeta )=0,\quad\ell<q;\quad \db_\zeta\Om^{13}_{(0,q)}( z,\zeta )=0.
\end{gather}

%We now derive a result analogous to \cite{MR986248}*{Lem. 13.6 $(iii)$, p.~122}.

We recall the following result.
\begin{lemma}[\cite{MR986248}*{p.~122, Lem. 13.6 $(iii)$}]\label{lemm:4.1} Let $1\leq q\leq n-2$. Let $(D^1,D_{r_2}^2,D_{r_3}^3)$ be   a  $(q+1)$ concave configuration with Leray maps $(g^1,g^2,g^3)$. Let $U^1=D_{r_3}^3\setminus\ov {D^1}$ and  $S_+^1=\pd D_{r_3}^3\setminus D^1$. 
\bpp\item
Suppose that $f \in C^1_{(0,q)}(\ov{ D_{r}^{123}})$ is $\db$-closed   on $D^{123}_r$. Then on   $D_{r}^{123}$,
\begin{align}
%\db_\zeta\Om^{13}_{0,q+1}( z,\zeta )=0, \quad q\geq1,\\
\label{dL2312}L_{13;q}^{13}f&=0.
%,\\
% \int_{S^1}\widehat\Om_{0,q}^{13}( z,\zeta )\wedge\db f (\zeta)=\widehat L_{1}^{13}\db f,
%L_{13;q}^{12}f&=-L_{13;q}^{23}f+\db L_{13;q-1}^{123}f+L_{13;q}^{123}\db f.
%\label{L1212}
\end{align}
\item If $(D^1,D_{r_2}^2,D_{r_3}^3)$ is a $(q+2)$-concave configuration and $0<q\leq n-3$, then \rea{dL2312} is %and \rea{L1212} are 
      valid on   $D_{r}^{123}$
for any   $f\in C_{(0,q)}^1(\ov{ D_{r}^{123}})$.
\epp
\ele

Note that  $\rho^2$ and $\rho^3$ are convex functions. Let $\hat g^0=g^0, \hat g^2=g^2$ and  $\hat g^3(z,\zeta)=\frac{\pd}{\pd\zeta}\rho^3$. 
Let $\hat\Om^{i}=\Omega^{\hat g^i}$ for $i=0,2,3$ and define $\hat\Om^{ij},\hat\Om^{ijk}$ analogously.   Then    the following  Bochner-Martinelli-Leray-Koppelman formula holds:  on $D_r^{23}$,
we have
\eq{BMLK}
g= \db\hat T_{D_{r}^{23}, \ell }g+\hat T_{D^{23}_{r},\ell+1}\db g,\quad 1\leq\ell\leq n,
\eeq
for $g\in C^1_{0,\ell}(\ov{D^{23}_{r}})$, where $\hat T_{D_r^{23},n+1}=0$ and $\hat T_{D_{r}^{23}, \ell }= R_{D_{r}^{23}, \ell -1 }-\hat L^{02}_{2, \ell -1}-\hat L^{03}_{3, \ell -1}+\hat L^{023}_{23, \ell -1}$ with
$$
\hat L^{0i}_{i,\ell }g:=\int_{S^{i}}\hat\Om^{0i}_{0,\ell} \wedge g,   \qquad \hat L^{023}_{23,\ell }g:=\int_{S^{23}}
\hat\Om^{023}_{0,\ell}\wedge g.
$$

The following is essential in \cite{MR986248}*{Lem.~13.7, p.~125}.
\begin{lemma} Let $1\leq q\leq n-2$. Let $(D^1,D^2,D^3)$ be  a  $(q+1)$-concave configuration.
 Let $L_{23,\bigcdot}^{23}$ be defined by \rea{defnLR}. 
For  $f\in C_{(0,q)}^1(S^{23})$, we have 
\ga{}\label{dbL23}
\db L^{23}_{23;q}f=L^{23}_{23;q+1}\db f\quad \text{on $ D_r^{23}$},\\
L^{23}_{23;q}f=\db \hat T_{ D_{(1-\theta)r}^{23};q}L^{23}_{23;q}f+\hat T_{D^{23}_{(1-\theta)r}, q+1}L_{23;q+1}^{23}\db f \quad \text{on $ D_{(1-\theta)r}^{23}$}.\label{L2323}
\end{gather}%\end{equation}
\ele
%\begin{rem}Note that $L^{23}_{23;q}f$ is defined on $D^{23}_r$ when $f\in C_{(0,q)}^1(\ov{D^{23}_r})$. However, we  have useful estimates for $L^{23}_{23;q}f$ only on shrinking domains $\ov{D^{23}_{(1-\theta)r}}$. Thus we use  $ \hat T_{ D_{(1-\theta)r}^{23};s}$ instead of $ \hat T_{ D_{r}^{23};s}$ in the lemma. 
%\end{rem}
\begin{proof}The form $\Om^{23}(z,\zeta)$ has no singularity in
 $z\in  D^{23}$ and $\zeta\in S^{23}\subset\pd D^{123}$.
 We have
$
\db_\zeta\Om_{(0,q+1)}^{23}+\db_z\Om_{(0,q)}^{23}=\Om_{(0,q+1)}^2-\Om_{(0,q+1)}^3.
$
By \re{type-2}, $\Om_{(0,q+1)}^2=0$. Thus $\int_{S^{23}}\Omega_{(0,q+1)}^2 \wedge f =0$.
  By  \re{type-3}, the $\zeta$-degree   of $f(\zeta)\wedge \Om^3_{(0,q+1)}( z,\zeta )$ is less than $ 2n-3$, which is less than $\dim (S^2\cap S^3)$.
This shows 
$
\int_{S^{23}}\Om_{(0,q+1)}^3 \wedge f =0,$ for $ q>0.
$
By Stokes' formula and $\pd S^{23}=\emptyset$, we obtain
$$
\db L^{23}_{23;q}f(z)=-\int_{S^{23}}\db_\zeta\Om_{(0,q+1)}^{23}( z,\zeta )\wedge f( \zeta )=\int_{S^{23}} \Om_{(0,q+1)}^{23}( z,\zeta )\wedge\db f(\zeta ).
$$
Finally, \re{L2323} follows from \re{BMLK}-\re{dbL23}. 
\end{proof}

%We now derive a homotopy formula on shrinking domains $D^{123}_{(1-\theta)r}$.
 Recall that the $D^{123}_{(1-\theta)r}$ and $ D^{23}_{(1-\theta)r}$ below are defined in \re{Diri}. 
\begin{thm} \label{cchf0} Let $1\leq q\leq n-3$. Let $(D^1,D_{r_2}^2 ,D_{r_3}^3 )$ be a  $(q+2)$-concave configuration defined by \rea{rho1-v}, \rea{rho2} and \rea{defD3}, for which the
 Leray mappings are defined by \rea{HL2pg81}, \rea{g02}, and \rea{defW-3}.     Let $f\in
C^0_{(0,q)}(\ov{D_{r}^{123}})
$ satisfy $\db f\in C^0(\ov{D_r^{123}})$. Then
$% \ga{}\label{tsqf+-cv0}
f=\db H^{(0)}_qf+  H^{(0)}_{q+1}\db f  %,\quad q>0,\\
%\label{tsqf+-cv}
% f=L_1^1f+L_2^2f+L_{12}^{12}f+  H_1\db  f.
$  %\end{gather}
holds on $D^{123}_{(1-\theta)r}$.  Here  $H^{(0)}_s=R_{D_r^{123},s-1}^{(0)}-L^{01}_{1,s-1}+\tilde L^{0123}_{s-1}$ with
\ga{} 
\tilde L^{0123}_{s-1}:=\sum_{i=1}^2L_{i3,{s-1}}^{0i3} +L_{12;{s-1}}^{123}-\hat T_{ D^{23}_{(1-\theta)r}, {s}}L^{23}_{23;{s}}.
\label{Hs2f}
\end{gather}
\eth 
\begin{rem}See~\cite{MR986248}*{Theorem 13.10 (i)} for a homotopy formula without shrinking domains, when $f,\db f$ are continuous. Our version will be useful for \rt{Linf}.
\end{rem}
\begin{proof} We begin with the estimate
  \ga{}\label{g3theta}
\RE\{g^3(z,\zeta)\cdot(\zeta-z)\}=\rho^3(\zeta)-\rho^3(z)
\geq r_3^2-(1-\theta)^2r_3^2\geq\theta r_3^2
\end{gather}  for $z\in D^{23}_{(1-\theta)r}$ and $\zeta\in S^{13}\cup S^{23}\cup S^{12}\subset\pd D^{123}$. Similar estimate holds for $\RE\{g^2(z,\zeta)\cdot(\zeta-z)\}$.  Together with the estimate \re{W1s-dist} for $g^1(z,\zeta)\cdot(z-\zeta)$, this implies that the integral kernels appearing in \re{Hs2f} are continuous in $z\in D^{123}_{(1-\theta)r}$.

 Recall the Bochner--Martinelli--Koppelman formula~\cite{MR1800297}*{p.~273} and a version for domains with piecewise $C^1$ boundary~\cite{MR986248}*{Def.~3.1, p.~46; Thm.~3.12, p.~53}:
 \aln %\label{BM}
  f(z)&=\db_z\int_{D^{123}}\Om_{(0,q-1)}^0(z,\zeta)\wedge f(\zeta)+\int_{D^{123}}\Om_{(0,q)}^0(z,\zeta)
 \wedge\db f(\zeta)\\
 \nonumber &\quad +\left\{\int_{S^1}+\int_{S^2}+\int_{S^3}\right\}\Om_{(0,q)}^0(z,\zeta)\wedge f(\zeta),
 \end{align*}
%when $ f$ is in $C^1(\ov D^{12})$.
where the orientations are chosen so that  $\pd D^{123}=S^1+S^2+S^3$.
Let us transform the three boundary integrals via   Koppelman's lemma. In what follows, we take $i=1,2$.
By  \re{kop2}, \re{Ca-type-1}  and the fact that $\Om_{(0,q)}^2=0$, we obtain  
\aln
\int_{S^i}\Om_{(0,q)}^0(z,\zeta)\wedge f(\zeta)&= \int_{S^i }\left(\db_\zeta\Om_{(0,q)}^{0i}(z,\zeta)+
\db_z\Omega^{0i}_{(0,q-1)}(z,\zeta)\right)\wedge f(\zeta)
\\
& =\int_{S^{i3}}\Om_{(0,q)}^{0i}(z,\zeta)\wedge  f(\zeta) -\int_{S^i}\Om_{(0,q)}^{0i}(z,\zeta)\wedge\db_\zeta f(\zeta)\\
&\qquad -
\db_z\int_{S^i}\Om_{(0,q-1)}^{0i}(z,\zeta)\wedge f(\zeta),
\end{align*}
where the third last identity is obtained by Stokes' theorem for $S^i$ with $\pd S^i=S^{i3}$.
Similarly, using $\pd S^3=S^{32}+S^{31}$, we obtain
\aln
\int_{S^3}
\Om_{(0,q)}^0(z,\zeta)\wedge f(\zeta)&= \int_{S^3 }\left(\db_\zeta\Om_{(0,q)}^{03}(z,\zeta)+
\db_z\Omega^{03}_{(0,q-1)}(z,\zeta)\right)\wedge f(\zeta)
\\
& =\int_{S^{32}+S^{31}}\Om_{(0,q)}^{03}(z,\zeta)\wedge  f(\zeta) -\int_{S^3}\Om_{(0,q)}^{03}(z,\zeta)\wedge\db_\zeta f(\zeta)\\
&\quad -
\db_z\int_{S^3}\Om_{(0,q-1)}^{03}(z,\zeta)\wedge f(\zeta).
\end{align*} 
Using $S^{i3}=-S^{3i}$, we get 
\aln{}&\left\{\int_{S^1 }+\int_{S^2 }+\int_{S^3 }\right\}\Om_{(0,q)}^0(z,\zeta)\wedge f(\zeta)\\
&\quad \quad \quad
=\sum_{i=1}^2\left\{\int_{S^{i3}}(\Om_{(0,q)}^{0i}-\Omega_{(0,q)}^{03})(z,\zeta)\wedge  f(\zeta) - \int_{S^i }\Om_{(0,q)}^{0i}(z,\zeta)\wedge\db_\zeta f(\zeta)\right.\\
&\quad\quad\qquad -
\left. \db_z\int_{S^i }\Om_{(0,q-1)}^{0i}(z,\zeta)\wedge f(\zeta)\right\}.
\end{align*}
Using $\Omega^{0i}-\Omega^{03}+\Omega^{i3}=\db_z\Omega^{0i3}+\db_\zeta\Omega^{0i3}$, we get
\aln{}
\int_{S^{i3}}(\Om_{(0,q)}^{0i}-\Omega_{(0,q)}^{03})(z,\zeta)\wedge  f(\zeta)&=\int_{S^{i3}}(\db_\zeta \Om_{(0,q)}^{0i3}+\db_z\Omega_{(0,q-1)}^{0i3})(z,\zeta)\wedge  f(\zeta)\\
&\quad-\int_{S^{i3}}\Om_{(0,q)}^{i3}(z,\zeta)\wedge  f(\zeta).
\end{align*}
The last term is zero for $i=1$ by \rl{lemm:4.1}. For $i=2$, by \re{L2323} we write 
$$%\eq{}
L^{23}_{23;q}f =\db\hat T_{ D_{(1-\theta)r}^{23};q}L^{23}_{23;q}f+\hat T_{  D^{23}_{(1-\theta)r}, q+1}L_{23;q+1}^{23}\db  f
$$%\eeq 
on $ D^{23}_{(1-\theta)r}$. By \re{checksign}, we have $\int_{S^{i3}}\db_z\Omega_{(0,q-1)}^{0i3}(z,\zeta)\wedge  f(\zeta)= \db_z\int_{S^{i3}}\Omega^{0i3}(z,\zeta)\wedge  f(\zeta)$.  
 Since $S^{i3}$ has no boundary, then 
$\int_{S^{i3}}\db_\zeta \Om_{(0,q)}^{0i3}(z,\zeta)\wedge f(\zeta) =  
\int_{S^{i3}} \Om_{(0,q-1)}^{0i3}(z,\zeta)\wedge\db_\zeta f(\zeta).
$
This shows that on $D^{123}_{(1-\theta)r}$
 \aln %\label{BM}
  f &=\db\int_{D^{123}}\Om_{(0,q-1)}^0 \wedge f +\int_{D^{123}}\Om_{(0,q)}^0 
 \wedge\db f  -\db\int_{S^1 }\Om_{(0,q-1)}^{01} \wedge f -\int_{S^1 }\Om_{(0,q)}^{01} \wedge \db  f \\ 
& -\db \hat T_{D_{(1-\theta)r}^{23};q}L^{23}_{23;q}f-\hat T_{ D^{23}_{(1-\theta)r}, q+1}L_{23;q+1}^{23}\db  f+\sum_{i=1}^2(\db L_{i3,q-1}^{0i3}f+L_{i3,q}^{0i3}\db f).\qedhere
 \end{align*}
 \end{proof}

We now use the extension operator. In the following we take $\cL E=\cL E_{D^{123}}$. Define 
$$
L_{1^+,\ell}^{01} f=\int_{S^1_+}\Omega^{01}_{(0,\ell)}\wedge\cL E f.$$

 \le{atfc}Assume that $f\in C^{1+\e}_{(0,q)}(\ov{D^{123}})$ with $\e>0$. On $D^{123}$, we have 
 \aln{}%\label{comm-term}
\db \int_{D^{123}}\Om_{(0,q-1)}^0 \wedge f &-\db\int_{S^1}\Om_{(0,q-1)}^{01} \wedge f+\int_{D^{123}}\Om_{(0,q)}^0 
 \wedge\db f\\
 \nonumber
&  -\int_{S^1}\Om_{(0,q)}^{01} \wedge \db  f =\db \tilde H_q^{(1)}f+\tilde H_{q+1}^{(1)}\db f
\nonumber
 \end{align*}
 with $\tilde H_s^{(1)}=
  R_{  D^{23},s-1  }
^{0}\cL  E + R_{D^3\setminus D^1,s-1}
^{01}[\db,\cL  E] -  L^{01}_{1^+,s-1}\cL E.
 $  
 \ele
\begin{proof}This is  in the proof of ~\ci{MR4866351}*{Prop.~2.5}, replacing the Stein extension by the Rychkov extension.  
\end{proof}

Recall that $U^1=D_{r_3}^3\setminus\ov {D^1}$. Thus, $D^{123}_r\cup U^1=D^{23}_{r}$.  By \rt{cchf0} and 
  \rl{atfc}, we have the following homotopy formula.
\th{cchf} Keep assumptions in \rta{cchf0}.
%Let $1\leq q\leq n-3$. Let $(D^1,D_{r_2}^2 ,D_{r_3}^3 )$ be a  $(q+2)$-concave configuration defined by \rea{rho1-v}, \rea{rho2}, and \rea{defD3}, together with
% Leray mappings defined by \rea{HL2pg81}, \rea{g02}, and \rea{defW-3}.  
    Let $f\in
C^{1+\e}_{(0,q)}(\ov{D_{r}^{123}})
$ with $\e>0$. Then
 \ga{}\label{tsqf+-cv}
f=\db H_qf+  H_{q+1}\db f  %,\quad q>0,\\
%\label{tsqf+-cv}
% f=L_1^1f+L_2^2f+L_{12}^{12}f+  H_1\db  f.
\end{gather}
 holds on $D^{123}_{(1-\theta)r}$
for $1/2<\theta<1$. Here for $s=q,q+1$, $H_s=H_s^{(1)}+H_s^{(2)}$ and
\gan{}
H_s^{(1)}=R^0_{D_r^{23},s}\cL E_{D^{123}}+R^{01}_{D_{r_3}^3\setminus D^1}[\db,\cL E|_{D^{123}}],\quad
H_{s}^{(2)}:=\tilde L^{0123}_{s-1}- L^{01}_{1^+,s-1}\cL E_{D^{123}}.
%=\sum_{i=1}^2L_{i3,s}^{0i3}f +L_{12;s-1}^{123}f-\hat T_{ D^{23}_{(1-\theta)r}, s}L^{23}_{12;s} f- %L^{01}_{1^+,s-1}\cL Ef.
%\label{Hs2f}
\end{gather*}
\eth 
%\begin{proof}Indeed, \re{tsqf+-cv0} and 
%  \rl{atfc} yields \re{tsqf+-cv}. %
%  
%Strictly speaking, the above computation is only valid when $\pd D\in C^3$, since the Koppelman lemma can be verified easily when all Leray maps $W_j\in C^2$. When $\pd D^i\in C^2$, one can still verify the integral formula on the domain $D^1\cap D^2$ by smoothing $g^j$. For instance, see for details.
% \end{proof}
 
Laurent-Thi\'ebaut and Leiterer~\cite{MR1621967}*{Prop.~0.7}  proved that no local homotopy formula with good estimates exists  near  $(q+1)$ concave boundary for $(0,q)$ forms.  Nevertheless,  the following result holds for $\db$-closed $(0,q)$-forms.
\th{cchf-closed} Let $1\leq q\leq n-2$. Let $(D^1,D_{r_2}^2 ,D_{r_3}^3 )$ be a  $(q+1)$-concave configuration.
Keep other conditions in \rta{cchf}.
% defined by \rea{rho1-v}, \rea{rho2}, and \rea{defD3}, together with
% Leray mappings defined by \rea{HL2pg81}, \rea{g02}, and \rea{defW-3}.  
    Assume that $f\in
C^{1+\e}_{(0,q)}(\ov{D_{r}^{123}})
$ with $\e>0$ is $\db$-closed.  On $D^{123}_{(1-\theta)r}$,  we have
$
f=\db H_qf$ for  $H_q$ in \rta{cchf}.
\eth

\setcounter{equation}{0}

\section{Exact $1/2$ gain estimates on shrinking domains}\label{h-space} 
In this section, we define the H\"older-Zygmund spaces and derive the exact $1/2$ gain estimates on shrinking domains.   

By
 a {\it bounded Lipschitz  domain} $D\subset\rr^n$, we mean that  there exist a finite open covering $\{U_i\}_{i=1}^N$ of $\ov D$, together with rigid affine transformations $A_i$, and positive numbers $\del_0,\del_1$, and $L$, such that for each $i$, the set   $A_i(\ov D\cap\ov{ U_i})$ is  defined by $\del_1\geq x_n\geq R_i(x')$ and  $x'\in[0,\del_0]^{n-1}$, where  $|R_i(\tilde x')-R_i(x')|\leq L|\tilde x'-x'|$. We say that a constant $C(D)$, depending on $D$, is {\it stable} under small perturbations of the Lipschitz domain $D$, if $C(\tilde D)$ can be chosen independent of $\tilde D$ if $A_i(\ov{\tilde D}\cap \ov{U_i})$ is defined by $\del_1\geq x_n\geq\tilde R_i(x')$, where  $|\tilde R_i(x')-R_i(x')|<\e$ and $|\tilde R_i(\tilde x')-\tilde R_i(x')|\leq (L+\e)|\tilde x'-x'|$ for some $\e>0$.
Similarly,  a constant $C(\nabla\rho,\nabla^2\rho)$ is  stable under small $C^2$ perturbations of $\rho^1$ if the corresponding constant $C(\nabla\tilde \rho,\nabla^2\tilde\rho)$ depends only on $\rho$, provided $\|\tilde\rho-\rho\|_{C^2}<\delta$ for some positive $\delta$ depending only on $\rho$. 

Let $\nn=\{0,1,\dots,\}$.
For $r\in[0,\infty)$, denote by $C^r(\ov D)$   the space of H\"older functions on $\ov D$ with the standard $C^r$-norm $\|\cdot\|_{C^r(D)}$.  
When $r=k+\beta$ with $k\in\nn$ and $0<\beta\leq1$, define $\Lambda^{r}(D)$  to be the set of functions $f$ on $D$ with finite norm
$$%\ga{}\label{firstLambdar}
\| f\|_{\Lambda^{r}(D)}:=\|f\|_{C^{k}(D)} +\sup_{
h\neq0,x\in D, x\pm h \in D} \f{|   f(x+h)+f(x-h)-2f(x)|}{|h|^\beta}.
$$%\end{gather} 
The reader is refer to \ci{MR4866351} for various equivalent H\"older-Zygmund norms. 

\le{abcd} Let $D\subset \rr^n$ be a bounded Lipschitz domain.  Set $|\cdot|_a=\|\cdot\|_{\Lambda^a(D)}$ and $\|\cdot\|_a=\|\cdot\|_{C^a(\ov D)}$. Then for functions $u,v$ on $D$,  
\aln{}
\|uv\|_a&\leq C_a(\|u\|_a\|v\|_0+\|u\|_0\|v\|_a),\quad  a\geq0,  
 \\ 
 %\label{Zcov}
|uv|_a&\leq C_a(|u|_a\|v\|_0+\|u\|_0|v|_a),\quad    \quad a>0.
\end{align*}
The $C_a$ is stable under small perturbations of 
 the Lipschitz domain $D$.
\ele

Let $D\subset \cc^n$ be a bounded Lipschitz  domain. Then the Rychkov   extension $\cL E=\cL E_D\colon \Lambda^a(D)\to \Lambda^a_0(\cc^n)$ satisfies 
$$
  \|\cL Ef\|_{\Lambda^a(\cc^n)}\leq C_a\|f\|_{\Lambda^a( D)},\quad a>0,
$$
where $C_a$ is stable under small perturbations of 
 the Lipschitz domain $D$. 
 
 \begin{thm}[convex configuration estimate]\label{conv-est} Let $r>0$ and $1\leq q\leq n-1$.
Let $(D^1,D_{r_2}^2)$ be an $(n-q)$-convex configuration. Let $D^{12}_r=D^1\cap D^2_r$.
The homotopy operators $H_q, H_{q+1}$ in \rta{hf-c} satisfies
\aln{}%\label{hqfr12-c}
\|H_s\var\|_{\Lambda^{r+1/2}(D^{1 2}_{(1-\theta)r_2})}&\leq \f{C_r( \nabla\rho^1,\nabla^2\rho^1)}{\theta^{3r+2r_0}}\|\var\|_{\Lambda^r(D^{1 2}_{r_2})},\quad \forall\theta\in(0,1). 
\end{align*}
Moreover,  $C_r(\nabla\rho^1,\nabla^2\rho^1)$ is stable under small $C^2$ perturbations of  $\rho^1$.
\end{thm}
\begin{proof}Recall from \re{hq1}-\re{nhq2-} that $H_s =H^{(1)}_s + H_s^{(2)}$ with  
\aln{}  
  H^{(1)}_s f&:=R_{  D^{2}; s-1}^0 \cL E_{D^{12}} f+R_{U^1;s-1 }^{01}[\db,\cL E_{D^{12}}] f,\\
H^{(2)}_sf&=L_{1^+;{s-1}
}^{01} \cL E_{D^{12}} f +L_{2;s-1}^{02} f+L_{12;s-1}^{012}f.
\end{align*}
Take a cut-off function $\chi\in C^\infty_0(D^{2}_{(1-\theta/3)r_2})$ with $\chi=1$ on $D^{2}_{(1-\theta/2)r_2}$, and decompose
$
 H^{(1)}_s f= H^{(1)}_s(\chi f)+ H^{(1)}_s( (1-\chi)f).
$
Then the estimate in~\cite{gong-shi-nn} says that 
$$
\|H_s^{(1)}(\chi \var)\|_{\Lambda^{r+1/2}(D^{1 2}_{r_2})}\leq C_r  \|\chi\var\|_{\Lambda^r(D^{1 2}_{r_2})}.
$$
We can choose $\chi$ such that $\|\chi\|_{C^r}\leq {C_r}{\theta^{-r-1}}$. Thus,  
$$
\|\chi\var\|_{\Lambda^r(D^{12}_{r_2})}\leq {C_r }{\theta^{-r-r_0}}\|\var\|_{\Lambda^r(D^{12}_{r_2})}.
$$
Since $(1-\chi)\var$ vanishes on $D_{(1-\theta/2)r}^{12}$ and
\ga{}\label{giest}
|g^i(z,\zeta)\cdot(\zeta-z)|\geq |\zeta-z|^2/C,\quad z\in D^{1}\cap D^2,\zeta\in D^2\setminus D^1, i=0,1,2,
\end{gather}
then we can obtain $$
\|H_s^{(1)}((1-\chi) \var)\|_{\Lambda^{r+1/2}(D^{12}_{(1-\theta)r_2})}\leq  C_r\theta^{-2r-r_0} \|(1-\chi)\var\|_{\Lambda^r(D^{12}_{r_2})}.
$$
The $H_s^{(2)}$ is given by boundary integrals on $S^1_+, S^2$ and $ S^{12}$ respectively. The estimate \re{giest}  also yields 
$\|H_s^{(2)}\var\|_{\Lambda^{r+1/2}(D^{12}_{(1-\theta)r_2})}\leq  C_r \theta^{-2r-r_0} \|\var\|_{\Lambda^r(D^{12}_{r_2})}.
$
We have obtained the desired estimate of $H_s$. 
\end{proof}
We now derive estimates for the concave case. We still have $H_s=H_s^{(1)}+H_s^{(2)}$, where $H_s^{(1)}$ has the same form as in the convex case. However, 
$$
H_{s}^{(2)}f=\sum_{i=1}^2L_{i3,s}^{0i3}f- L^{01}_{1^+,s-1}\cL Ef +L_{12;s-1}^{123}f-\hat T_{ D^{23}_{(1-\theta)r}, s}L^{23}_{12;s} f.
$$
By \re{g3theta}, we have $
\RE\{g^3(z,\zeta)\cdot(\zeta-z)\} \geq\theta r_3^2$ for $z\in D^{23}_{(1-\theta)r}$ and $\zeta\in S^{13}\cup S^{23}\cup S^{12}\subset\pd D^{123}_r$. Some kernels in $H^{(2)}_s$   involve   first-order $z$-derivatives of $\nabla\rho^1(z)$.
As in the convex case, we can obtain 
$$%\eq{H232}
\|H^{(2)}_sf\|_{\Lambda^{r+1/2}(D^{123}_{(1-\theta)r})}\leq {C_r(\nabla\rho^1 )}{\theta^{-r-r_0}}\|\rho^1\|_{\Lambda^{3/2}}\|f\|_{\Lambda^\e}.
$$%\eeq
The kernels of $H^{(1)}f$ involve the second-order $z$-derivatives  $ \nabla\rho^2(z)$.
Then by ~\cite{gong-shi-nn}*{Thm. 5.12}, we have  
$$%\eq{H152}
\|H^{(1)}_sf
\|_{\Lambda^{r+\del}(D^{123}_{(1-\theta)r})}\leq {C_r }{\theta^{-r-r_0}}\left(\|\rho^1\|_{\Lambda^{r+2+\del}}\|f\|_{\Lambda^\e}+ \|f\|_{\Lambda^r}\right).
$$%\eeq
Therefore, we have obtained part $(a)$ below.
\begin{thm}[concave configuration estimate]\label{concave-est} Let $\e>0,r>0$ and $1\leq q\leq n-3$.  
Assume that $\rho^1\in C^{2+\e}$. The following hold.
\bpp
\item  The operators $H_q, H_{q+1}$ in Theorem~$\ref{cchf}$ satisfy
\al{}
\label{hqfr12}
\|H_sf\|_{\Lambda^{r+\del}(D^{123}_{(1-\theta)r})}&\leq \f{C_{r,\e}}{\theta^{3r+r_0}} (\|\rho^1\|_{\Lambda^{r+2+\del}}\|f\|_{\Lambda^\epsilon({D_{r}^{123}})}+  \|f\|_{\Lambda^r(D^{123}_{r})})
\end{align}
for $0\leq\delta\leq1/2$ and $0<\theta<1$.  Furthermore,   $C_{r,\e}=C_{r,\e}(\nabla\rho^1,\nabla^2\rho^1)$ is stable under small $C^2$ perturbations of $\rho^1$;
  the same estimate holds for the $\db$ solution operator $H_q$  in \rta{cchf-closed}.
  \item 
The homotopy formula \rea{tsqf+-cv} in Theorem~$\ref{cchf}$ holds for $f\in \Lambda^{\e}_{(0,q)}(D^{123})$ when $\db f\in \Lambda^{\e}_{(0,q)}(D^{123})$, and the $\db$-solution operator $H_q$  in \rta{cchf-closed} is valid for all $\db$-closed  $f\in\Lambda^{\e}_{(0,q)}(D^{123})$ and satisfies \rea{hqfr12}.
\epp
\end{thm}
\begin{proof}We verify $(b)$ using $(a)$.  We know that  $f=\db H_qf+H_{q+1}\db f$ when $f\in C^{1+\e}$. Let us write $H_q$ as $H_{D^{123}_r}$ to indicate the dependent on the domain. Now assume $f\in \Lambda^r_{(0,q)}(D^{123})$ with $\db f\in\Lambda^r$. We  find a sequence $f_j\in C^{\infty}_{(0,q)}(D^{123})$ such that $f_j $ and $\db f_j $ converge to $f,\db f$ in $\Lambda^{r'}(D^{123})$ as $j\to\infty$ for any $r'<r$. Furthermore, $$\|f_j\|_{\Lambda^{r}(D^{123})}\leq C_r\|f\|_{\Lambda^{r}(D^{123})}, \quad \|\db f_j\|_{\Lambda^{r}(D^{123})}\leq C_r\|\db f\|_{\Lambda^{r}(D^{123})}.$$
Then $H_qf=\lim H_qf_j$ and $H_{q+1}\db f=\lim H_{q+1}\db f_j$ gives us \rea{tsqf+-cv}.
Assume now $f$ is $\db$-closed. Set  $L_t(z)=z+tie_{q+2}$ with $t>0$, where $e_{q+2}$ is the $(q+2)$-th   unit vector of $\cc^n$. Fix $\theta\in(0,1/2)$. When $0<t<t_\theta$, we have $L_t(D^{123}_{(1-\theta)r})\subset D^{123}_r$. Using a  mollifier, we find $\db$-closed $f_j\in C^\infty(D^{123}_{(1-\theta)r})$ such that   $f_j$ converges to $f$ in $\Lambda^{r'}(D^{123}_{(1-\theta)r})$-norm for any $r'<r$, whereas $\|f_j\|_{\Lambda^{r}(D^{123}_{(1-\theta)r})}\leq C_r\|f\|_{\Lambda^{r}(D^{123}_{r})}$. We have $f_j=\db H_{D^{123}_{(1-\theta')r},q}f_j$ on $D^{123}_{(1-\theta)r}$ for $0<\theta'<\theta/2$.
Letting $j\to\infty$, we obtain $f=\db H_{D^{123}_{(1-\theta')r},q}f$ on $D^{123}_{(1-\theta)r}$ and
\eq{Hth''}
\|H_{D^{123}_{(1-\theta')r},q}f\|_{\Lambda^{r+\del}(D^{123}_{(1-\theta)r})}\leq   \f{C_{r,\e}}{\theta^{3r+r_0}} (\|\rho^1\|_{\Lambda^{r+2+\del}}\|f\|_{\Lambda^\epsilon }+  \|f\|_{\Lambda^r }).
\eeq
By the Arzel\`{a}-Ascoli theorem,  $H_{D^{123}_{(1-\theta'_j)r},q}f$ converges to $H_{D^{123}_{r},q}f$ in   $\Lambda^{r'+\delta}(D^{123}_{(1-\theta)r})$ norm for a sequence $\theta'_j\to0$ for any $r'<r$. Then  \re{Hth''} yield
\re{hqfr12} for $H_q$.
\end{proof}

We conclude this section with local homotopy formulas when $f$ and $\db f$ are   $L^\infty$.  

\begin{thm}[$C^{1/2}$ and $C^0$ estimates] \label{Linf}Let $(D^1,D^2_{r_2})$ be an $(n-q)$ convex configuration as in \rta{conv-est}. Then there exists a homotopy formula 
$
f=\db H_{q}f+H_{q+1}\db f
$
   on $D_{r_2}^{12}$, for $f\in L_{(0,q)}^\infty(D^{12}_{r_2})$ with $\db f\in L^\infty(D_{r_2}^{12})$. Moreover,
$$%\eq{}
\|H_sg\|_{C^{1/2}(D^{12}_{(1-\theta){r_2}})}\leq {C}{\theta^{-c_0}}
% \f{C}{\theta^{c_0}}
\|g\|_{L^\infty(D^{12}_{r_2})}.
$$%\eeq
Let $(D^1,D^2_{r_2}, D^3_{r_3})$ be a $(q+2)$ concave configuration as in \rta{cchf}. Then there exists a homotopy $
f=\db H_{q}f+H_{q+1}\db f
$ on $D^{123}_{(1-\theta)r}$, for $f\in L_{(0,q)}^\infty(D^{123}_r)$ with $\db f\in L^\infty(D_r^{123})$. Furthermore,
$$%\eq{}
\|H_sg\|_{C^{0}(\ov{D^{123}_{(1-\theta)r}})}\leq {C}{\theta^{-c_0}}
%\f{C}{\theta^{c_0}}
\|g\|_{L^\infty(D^{123}_r)}.
$$%\eeq
\end{thm}
\begin{proof}Let us consider only the concave case, since the strongly pseudoconvex case is simpler and already well known. Recall that  
$%\eq{fdbH}
f=\db H_qf+H_{q+1}\db f
$  %\eeq
holds in the sense of distributions on $D^{123}_{(1-\theta)r}$
whenever $ f,\db f$ are continuous on $\ov{D^{123}}$.
We expand  
$
H_sg=\sum_{j=1}^m a_{s,j} H_{s,j}g,
$
where $a_{s,j}(z)$ are polynomials in $\nabla_z^2\rho$ whose coefficients are smooth functions in $z$. Thus, $a_{s,j}$ are continuous on $\ov {D^1}$. Each $\cL H_{s,j}$, of which the kernels involve only $\nabla\rho^1$,  gains $1/2$ derivative 
\eq{Hsjg}
\|H_{s,j}g\|_{C^{1/2}(D^{123}_{(1-\theta)r})}\leq{C_1}{\theta^{-r_0}}\|g\|_{L^\infty }
\eeq
with $\|\cdot\|_{L^\infty }:=\|\cdot\|_{L^\infty(D_r^{123})}$, which is the  estimate of Henkin-Romanov~\cites{MR0293121,MR986248}. Suppose that $f$ is a $(0,q)$ form on $D_r^{123}$ such that $f,\db f$ are $L^\infty$ on $D^{123}$. Then we   find a sequence $f_\ell \in C_{0,q}^{\infty}(\ov{ D_r^{123}})$ such that $f_\ell,\db f_\ell$ converge to $f,\db f$ in $L^1$ norm on $D^{123}$. Moreover,
\eq{fL}
\|f_\ell\|_{L^\infty}\leq C_2\|f\|_{L^\infty}, \quad  \|\db f_\ell\|_{L^\infty}\leq C_2\|\db f\|_{L^\infty}.
\eeq
Using \re{Hsjg}-\re{fL}, applying the Arzel\`{a}-Ascoli theorem  to  $H_{q,j}f_\ell,H_{q+1,j}\db f_\ell$, and passing to a subsequence if necessary,  $H_{q,j}f_\ell,H_{q+1,j}\db f_\ell$ converge in $L^\infty(D^{123}_{(1-\theta)r})$ norm to $u_{j},v_{j}$. Moreover,
$$%\eq{Hsjg+}
\|u_j\|_{C^{1/2}(D^{123}_{(1-\theta)r})}\leq{C_2C_1}{\theta^{-r_0}}\|f\|_{L^\infty}, \quad
\|v_j\|_{C^{1/2}(D^{123}_{(1-\theta)r})}\leq{C_2C_1}{\theta^{-r_0}}\|\db f\|_{L^\infty}.
$$%\eeq
We can define $H_{q}f=\sum a_{q,j}u_j$ and $H_{q+1}\db f=\sum a_{q+1,j}v_j$.   The above-mentioned convergence of $f_\ell,H_{q,j}f_\ell,H_{q+1,j}\db f_\ell$ implies that $f=\db H_qf+H_{q+1}\db f$ on $D^{123}_{(1-\theta)r}$ in the sense of distributions.
%Using \re{Hsjg},  the Arzel\`{a}-Ascoli theorem, and passing to a subsequence if necessary, we may find $u,v$ such that
%\gan{}
%\lim_{\ell\to\infty}\| H_{q,j}f_\ell -u\|_{C^0(D^{123}_{(1-\theta)r})}=0, \quad \|u\|_{C^{1/2}(D^{123}_{(1-\theta)r})}\leq\f{ C_1C_2}{\theta^{r_0}}\|f\|_{L^\infty( D^{123})},\\
%\lim_{\ell\to\infty}\|  H_{q+1,j}\db f_\ell -v\|_{C^0(D^{123}_{(1-\theta)r})}=0, \quad \|v\|_{C^{1/2}(D^{123}_{(1-\theta)r})}\leq\f{ C_1C_2}{\theta^{r_0}}\|\db f\|_{L^\infty(D^{123})}.
%\end{gather*}
%Define $H_qf=u$ and $H_{q+1}\db f=v$. Then \re{fdbH} holds in the sense of distributions because $f_\ell$ converges to $f$ in $L^1$. Moreover, \re{Hsjg} continue to hold for $f$, with the constant $C_1$   replaced by $C_1C_2$. 
\end{proof}
%\begin{rem}
%When $q=1$,  we have additionally $\db H_qf=f-H_{q+1}\db f\in L^\infty(D^{12}_{(1-\theta)r})$ for the $(n-q)$ convex  case; consequently, the function $H_1f$ must be in $ W^{1,2}_{loc}(D^{12}_{(1-\theta)r}\setminus\pd D^1)$. Analogously, we have   $H_1f\in W^{1,2}_{loc}(D^{123}_{(1-\theta)r}\setminus \pd D^1)$ for the $2$-concave case, when $\pd D^1$ is merely $C^2$. 
%\end{rem}

\setcounter{thm}{0}\setcounter{equation}{0}
\section{Integrability of complex vector bundles}\label{sec:NM}
In Section 2, we have seen that \rt{vbint-} follows from the following result.
\begin{thm}\label{vbint} 
Let $M$ be a domain in $\mathbb C^n$ with $\pd M\in C^2$. Assume that    the Levi-form of $M$ has
either $(n-1)$ positive or at least $3$ negative Levi eigenvalues at $\zeta_0\in\pd M$. Let $U$ be an open set in $\ov M$ and $\zeta_0\in U$.    Assume that $\om$ is a $k_0\times k_0$ matrix of $(0,1)$ forms on $U$ such that $\db\om=\om\wedge\om$ holds on $U\setminus\pd M$ in the sense of distributions. Assume that $\om$ is in $L^\infty(U)$ $($resp. $\Lambda^r(U)$ with $r>0)$. Then there is an invertible matrix $A$   of class $\Lambda^{1/2}$ $($resp. $\Lambda^{r+1/2})$ on some neighborhood $U'$ of $\zeta_0$ in $\ov M$ such that 
$
\db A+A\om=0
$
in the sense of distributions on $U'\setminus\pd M$.
\end{thm} 
\begin{proof}The proof is based on a KAM-type iteration method. Such an approach was developed by Webster~\ci{MR999729} for an interior version of Newlander-Nirenberg theorem. It  was also used in~\cites{MR2742034,MR2829316} to establish the integrability of CR vector bundles on strictly pseudoconvex hypersurfaces in $\cc^n$ for $n\geq4$. 

\medskip

\noindent{\bf The case of $L^\infty$ connections.} For a boundary point  $\zeta_0\in\pd M$, we can find a biholomorphic mapping $\psi$, with $\psi(\zeta_0)=0$, to get    a convex configuration $(D^1,D_{r_2}^2)$  or a concave configuration $(D^1,D_{r_2}^2,D_{r_3}^3)$ such that
$
\var=\db P_{D_{r}}\var+Q_{D_{r}}\db\var
$
on $D_{(1-\theta)r}$ for $(0,1)$-forms $\var$ given on $D_r$, where $D_{r}:=D^{1}\cap D^2_{r_2}$ for $(n-1)$ convex case and $D_{r}:=D^{1}\cap D^2_{r_2}\cap D^3_{r_3}$ for $3$-concave case. 
 Moreover, we have the $C^0$ estimate
\eq{pC0}
\|P_{D_{r}}\var\|_{C^0(\ov D_{(1-\theta)r})}\leq {C_0}{\theta^{-c_0 }}\|\var\|_{L^\infty( D_r)};
\eeq
 same estimate holds for $Q_{D_{r}}$.  The problem is local and hence it suffices to prove the theorem replacing $M,\zeta_0,E, D$ by $D^1,0,\psi_*E, \psi_*D$. We still denote the latter by $E,D$.

Let $\om_0$ be given on  $D_{0}$ and satisfy $\db\om_0=\om_0\wedge\om_0$. We define  sequences
$$
r_{k+1}=(1-\theta_k)r_k, \ \theta_k=(k+2)^{-2}, \  D_k:=D_{r_k}, \  P_k=P_{D_k},\quad  Q_k=Q_{D_k}, \quad k\geq0. 
$$
We want to construct a sequence of invertible matrices $A_k=I+B_k$ defined on $D_{k+1}$. On $D_{k+1}$, 
  we define $\om_{k+1}$ inductively by
\eq{omk+1}
\om_{k+1}=(\db A_k)A_k^{-1}+\om_k=\db B_k+\om_k+(\db B_k)\tilde B_k,
\eeq
where $\tilde B_k=A_k^{-1}-I$. 
Since $\om_k=\db P_k\om_k+Q_k\db\om_k$ and $\db\om_k=\om_k\wedge\om_k$, we choose  $B_k=-P_{k}\om_k$. Substituting into \re{omk+1}, we obtain
\aln{}
\om_{k+1}&= Q(\om_k\wedge\om_k)+(\om_k-(Q_k(\om_k\wedge\om_k))\tilde B_k.\end{align*}
We define $\del_{k+1}:=\del_k^{3/2}$, where $0<\del_0<\f{1}{4k_0}$ and $\del_0$ is sufficiently small.  Note that $\del_k=\del_0^{(3/2)^k}$ decrease rapidly to $0$ as $k\to\infty$.
We will show that $A_k$ are admissible and
\eq{Bkc0}
\|\om_k\|_{L^\infty(\ov D_k)}<\delta_k,  \quad \|B_k\|_{C^0(\ov D_{k+1})}=\|P_k\om_k\|_{C^0(\ov D_{k+1})}<C_0\theta_k^{-c_0}\del_k,
\eeq
  provided that the initial $\om_0$ satisfies  
\eq{}\label{om0}
\|\om_0\|_{L^\infty}<\delta_0.
\eeq

For convenience, we omit explicit mention of the domains on which the norms are taken.
From \re{Bkc0}, we know that $I+B_k$ is invertible on $D_{k+1}$, and  furthermore
$$
\|\tilde B_k\|_{C^0}\leq 2\|B_k\|_{C^0}.
$$ 
Since $\db A_k=\db B_k=\om_k-Q(\om_k\wedge\om_k)\in L^\infty(D_{k+1})$, the matrix $A_k$ is admissible on $D_{k+1}$. Therefore the formal integrability condition $\db\om_{k+1}=\om_{k+1}\wedge\om_{k+1}$ is satisfied and the iteration may continue.
From the expression for $\om_{k+1}$, we obtain the estimate $\|\om_{k+1}\|_{C^0}\leq C_0\theta_k^{-c_0}\del_k^2+C_0 \theta_k^{-2c_0}\del_k^2<\del_{k}^{3/2}.
$
 This verifies \re{Bkc0} for $\om_{k+1}$, and hence  for all $k$,
provided \re{om0} is satisfied. 
Set $\hat A_k=A_kA_{k-1}\cdots A_0$.  Then $\hat A_{\infty}=\lim_{k\to\infty}\hat A_{k}$ is continuous and invertible on $D_{r_\infty}$ for $r_\infty=\lim r_k$,
and $\db A_\infty +\hat A_\infty\om_0 =0$. To apply \rp{d2=0}, we must show that $\db \hat A_\infty\in L^2_{loc}(D_{r_\infty})$, which follows from
$\hat A_k-\hat A_{k-1}=B_kA_{k-1}\cdots A_0$ and
\aln{}
&\|\db B_k\|_{L^\infty}\leq \|\om_k\|_{L^\infty}+\|Q(\om_k\wedge\om_k)\|_{L^\infty}\leq \del_k+{C}{\theta_k^{-c_0}}\del_k^{2}<2\del_k,\\
&\|\db (\hat A_{k+1}-\hat A_k)\|_{L^\infty}\leq k_0\|\db B_k\|_{L^\infty}\prod_{j=0}^{k-1} (1+k_0\| B_j\|_{L^\infty})\\
&\quad +k_0\| B_k\|_{L^\infty}\sum_{j=0}^{k-1}k_0\|\db B_j\|_{L^\infty}\prod_{0\leq \ell\leq  k-1,\ell\neq j} (1+k_0\| B_\ell\|_{L^\infty})  \leq C^{k+1} \del_k^{3/2}.\nonumber
\end{align*} 

Therefore, it remains to achieve
$\|\om_0\|_{L^\infty(D_{r_0})}<\del_0$ for  any given $\del_0>0$. As in \ci{MR2742034}, we can achieve this by a non-linear   dilation at $x$ as follows.  
Recall that $D^1\subset \Delta_{r_0}\times\Del^{n-1}_{r_0/{C_n}}$, with $z_3,z_1$ switched, is defined by
\begin{gather}\label{rho1-v-5}
\rho^1(z)=-y_{1}-|z_1|^2-|z_2|^2-|z_{3}|^2+\la_{4}|z_{4}|^2+
\cdots+\la_n|z_n|^2+R(z)<0.
\end{gather}
 Consider the dilation $z=L(\tilde z)=(\e^2 \tilde z_1,\e \tilde z')$. Let $\tilde D$ be the induced connection for the pull-back bundle  $\tilde E=L^*E$ over   $\tilde D^1=L^{-1}(D^1)\subset\Delta_{r_0}\times\Del^{n-1}_{r_0/{C_n}}$.
Set  $\tilde e=L^*e=e\circ L$. Thus
$$
\tilde D\tilde e=L^*\om\otimes \tilde e,\quad \tilde D(\tilde f\tilde s)=d\tilde f\otimes s+\tilde f\tilde D\tilde s
$$
are well-defined as $L^*d=dL^*$. 
Note that $\tilde D^1\subset \Delta_{r_0}\times\Delta^{n-1}_{r_0/{C_n}}$ is given by
\begin{gather*}
\tilde\rho^1(z)=-y_{1}-|z_2|^2-|z_{3}|^2+\la_{4}|z_{4}|^2+
\cdots+\la_n|z_n|^2+\tilde R(z)<0.
\end{gather*}
We can find a biholomorphism $\hat\psi$ that is independent $\e$ such that $\hat D^1:=\hat\psi^{-1}(\tilde D^1)$ still has the form \re{rho1-v-5}.  The connection form for $\hat E=\hat\psi^*\tilde E$ with connection $\hat D=\hat\psi^*\tilde E$ is given by $\hat\om=\hat\psi^*L^* \om$. Since $L,\hat\psi$ are biholomorphic, we have $\hat\om^{(0,1)}=\hat\psi^*L^*\om^{(0,1)}$.  Since $\om^{(0,1)}$ has degree $1$ (i.e. homogeneous of degree 1 in $d\ov z$),   the non-isotropic dilation $L$ yields
$$%\eq{hatom}
\hat\om^{(0,1)}=O(\e), \quad \|\hat\om^{(0,1)}\|_{L^\infty(D_{r_0})}<\delta_0.
$$%\eeq 
when $\e$ is sufficiently small. 
We emphasize that with the non-linear dilation $L\circ\hat\psi$, the constant $C_0$ in \re{pC0}  depends only on $n,r_0$.

\noindent{\bf The case of strongly pseudoconvex boundary.}   Assume  that $\om_0\in \Lambda^a$ with $\infty>a>0$. Then using \rt{conv-est}, we can replace \re{Bkc0} by
$$%\eq{Bkc0-a}
\|B_k\|_{\Lambda^{a+1/2} }<C_0\theta_k^{-c_0}\del_k, \quad \|\om_k\|_{\Lambda^a }<\delta_k.
$$%\eeq
Using the non-linear dilation, we can still achieve $|\om_0|_{\Lambda^{a}(D_0)}<\delta_0$. The same proof yields $A\in \Lambda^{a+/2}(D_\infty)$. 

Assume now that $\om_0\in C^\infty(D_0)$. We want to show that the sequence $A_k$ constructed for a finite $a$, say $a=1$, yields $A_\infty\in C^\infty(D_\infty)$. Indeed, fix $0<\ell<\infty$.     Using  $\|uv\|_{\Lambda^\ell}\leq C_\ell(\|u\|_{\Lambda^\ell}\|v\|_{C^0}+\|u\|_{C^0}\|v\|_{\Lambda^\ell})$,
we    obtain
$$
\|\om_{k+1}\|_{\Lambda^\ell}\leq C_\ell\theta_k^{-c_0(\ell+1)}\|\om_{k}\|_{C^0}\|\om_{k}\|_{\Lambda^\ell}\leq C_\ell^{k+1}\|\om_0\|_{\Lambda^\ell}\prod_{j=0}^k\theta_j^{-c_0(\ell+1)}\del_j.
$$
We have $L_{k+1}/L_k=C_\ell\theta_k^{-c_0(\ell+1)}\del_k<\del_{k-1}$ when $k>k_\ell$. Therefore
$\|\om_{k+1}\|_{\Lambda^\ell}
\leq C_\ell'\del_{k-1}$ for $ k>k_\ell.
$
We also have $\|B_{k+1}\|_{\Lambda^\ell}\leq C_\ell\theta_k^{-c_0(\ell+1)}\|\om_k\|_{\Lambda^\ell}\leq \del_{k-2}$ for $k>k_\ell'$. Now it is easy to verify that $A_\infty =\prod_{j=0}^\infty(I+B_j)$ is in $\Lambda^\ell(D_\infty)$. 
\medskip

\noindent{\bf The case of $3$-concave boundary.} We want to show that the solution $A_\infty$ in the $L^\infty$ case near a $3$-concave boundary point of $M$ actually belongs to $\Lambda^{a+1/2}(D_\infty)$
when $\om_0\in\Lambda^a(D_0)$ additionally. The proof relies on the Hartogs's extension theorem. This technique was developed in~\ci{MR4866351} for $\db$-solutions for $(0,1)$ forms near $2$-concave boundary point and was later used  in~\ci{gong-shi-nn} in the context of Newlander-Nirenberg theorem near a $3$-concave boundary point. 

We first remark that using \rt{concave-est} instead of \rt{conv-est},
 the argument for the first two cases already shows that if $\pd M\in \Lambda^{a+5/2}$ then we can construct $A_\infty\in\Lambda^{a+1/2}(D_{r_\infty})$ for $\om_0\in\Lambda^a(D_0)\cap L^\infty(D_0)$  $0\leq a\leq\infty$. Let us take any $3$-concave boundary point $x\in D_{r_\infty/C_*}\cap \pd D_0$, where $C_*$ is sufficiently large.
We find a biholomorphic map $\psi_x$ such that $\psi_x(x)=0$ and $D^1_{x}:=\psi_x(D_{r_\infty})\cap (\Delta_{r_0}\times\Delta^{n-1}_{r_0/{2}})$ is given by
\begin{gather*}%\label{rho1-v-5+}
\rho^1(z)=-y_{1}-|z_1|^2-|z_2|^2-|z_{3}|^2+\la_{4}|z_{4}|^2+
\cdots+\la_n|z_n|^2+R(z)<0
\end{gather*}
with  $\|R\|_{C^2}\leq \del_0^*$. 
Consider a smooth domain $\hat D^1\subset  \Delta_{r_0}\times\Del^{n-1}_{r_0/{2}}$ defined by
\begin{gather*}%\label{rho1-v-5+}
\hat \rho^1(z)=-y_{1}-|z_1|^2-|z_2|^2-|z_{3}|^2+\la_{4}|z_{4}|^2+
\cdots+\la_n|z_n|^2+\e|z|^2<0,
\end{gather*}
where $C_a$ is independent of $x$. When $\e>0$, depending on $\del_0^*$, is sufficiently small, we have $\hat D^1\subset D_x^1$, while $\pd\hat D^1$ intersects $\pd D_x^1$ only at $0$. Recall that $D''e_0=\om_0 e_0$. 
Therefore, we can find $\hat A_x\in \Lambda^{a+1/2}(\hat D^{123}_{r_\infty})$ with $\hat D^{123}_{r_\infty}= \hat D^1\cap D^2_{c_*r_2}\cap D^3_{c_*r_3}$ with $c_*=\prod(1-\theta_k)$ such that 
$
\db \hat A_x +\hat A_x\hat\om_x=0$ on $\hat D^{123}_{r_\infty}$,
where $\hat\om_x=(\psi_x^{-1})^*\om$. Moreover,  for some constant $C_a'$ independent of $x$,
\eq{uniform-x}
\|\hat A_x\|_{\Lambda^{a+1/2}(\hat D^{123}_{r_\infty}) }<C'_a.
\eeq
Let $ D^{123}_{x,r_\infty/2}=\psi_x^{-1}(\hat D^{123}_{r_\infty/2})$ and $A_x=\psi_x^*\hat A_x$. 
On $ D^{123}_{x,r_\infty/2}\subset D^{123}_{r_\infty}$, we also have $\db A_\infty +A_\infty\om_0=0$.
By \re{hatomA},   $\db (A_xA_\infty^{-1})=0$ on $ D^{123}_{x,r_\infty/2}$. By Hartogs's extension theorem,   the holomorphic matrix $A_xA_\infty^{-1}$ defined  on $D^{123}_{x,r_\infty/2}\setminus\pd  D^1$ extends holomorphically across the portion of the boundary of $D^{123}_{r_\infty}$ that meets $ \pd D^1$. Hence $A_\infty\in\Lambda^{a+1/2}( D^{123}_{x,r_\infty})$. Furthermore, $\|A_xA_\infty^{-1}\|_{L^\infty( D^{123}_{x,r_\infty})}<C$ where $C$ is independent of $x$. The Hartogs's extension via Cauchy formula then yields \eq{uniform}
\|A_xA_\infty^{-1}\|_{\Lambda^{a+1/2}( D^{123}_{x,r_\infty/2}) }<C_a.
\eeq
Combining \re{uniform-x}-\re{uniform}, we obtain 
\eq{unif+}
\|A_\infty\|_{\Lambda^{a+1/2}(D^{123}_{x,r_\infty/2}) }\leq C_a''C_aC_a'\|A_\infty\|_{C^0}.
\eeq

As $x$ moves in $D_{r_\infty/C_*}\cap \pd D_0$, the union of $ D^{123}_{x,r_\infty/2}$ contains a neighborhood $\hat U$ of $0$ in $\ov{D_0}$. Let  $a+1/2=k+\beta$ with $k\in\nn$ and $0<\beta\leq1$. Then \re{unif+} implies  $A\in C^k(\hat U)$. Using \re{unif+} again, we  estimate $\|A\|_{\Lambda^\beta(\hat U)}\leq C_a'''\|A_\infty\|_{C^0}$, possibly after shrinking $\hat U$; see~\ci{MR4866351}*{Prop. 7.1} for details. 
\end{proof}
 \begin{rem}
The above argument can be modified to show that when $x\in M$ and $a>0$, we can find $A\in\Lambda^{a+1}$   in a neighborhood of $x$ in $M$ for $\om_0\in\Lambda^a$, by using the Leray-Koppelman homotopy formula on unit ball of which the homotopy operator gains a full derivative. Moreover, for $\om_0 \in L^\infty$, We can still find $A\in \Lambda^\e$ for a fixed $\e<1$  in a neighborhood of $x$ in $M$.
\end{rem}

  % \nocite{*}

\newcommand{\doi}[1]{\href{http://dx.doi.org/#1}{doi:#1}}
\newcommand{\arxiv}[1]{\href{https://arxiv.org/pdf/#1}{arXiv:#1}}

\bibliographystyle{alpha}

%\bibliography{../g-bib/master,
%aq-domain.bbl}


\end{document}